\documentclass[10pt,a4paper]{amsart}
\usepackage[final]{optional}
\usepackage[british,american]{babel}

\usepackage{mathrsfs}
\usepackage{amsfonts, amsmath, wasysym}

\usepackage{enumitem}
\usepackage{amssymb,amsthm}

\usepackage{
latexsym,
nicefrac,
}
\usepackage[usenames]{color}

\usepackage{url}

\definecolor{darkgreen}{rgb}{0,0.5,0}
\definecolor{darkred}{rgb}{0.7,0,0}
\usepackage[colorlinks, 
citecolor=darkgreen, linkcolor=darkred
]{hyperref}
\usepackage{esint}
\usepackage{bibgerm}
\usepackage[normalem]{ulem}

\theoremstyle{plain}
\newtheorem{lemma}{Lemma}[section]
\newtheorem{thm}[lemma]{Theorem}
\newtheorem{prop}[lemma]{Proposition}
\newtheorem{cor}[lemma]{Corollary}

\theoremstyle{definition}

\newtheorem{rmk}[lemma]{Remark}

\def\blbox{\quad \vrule height7.5pt width4.17pt depth0pt}

\newcommand{\gcmt}[1]{\opt{draft}{\textcolor[rgb]{0,0.5,0}{
$\LHD$ #1 $\RHD$\marginpar{\blbox}}}}

\numberwithin{equation}{section}

\newcommand{\al}{\alpha}

\newcommand{\de}{\delta}
\newcommand{\om}{\omega}
\newcommand{\Om}{\Omega}

\newcommand{\La}{\Lambda}

\newcommand{\si}{\sigma}

\newcommand{\Si}{\Sigma}
\renewcommand{\th}{\theta}

\newcommand{\R}{\ensuremath{{\mathbb R}}}
\newcommand{\N}{\ensuremath{{\mathbb N}}}

\newcommand{\C}{\ensuremath{{\mathbb C}}}

\newcommand{\diam}{\text{diam}}

\newcommand{\upto}{\uparrow}

\newcommand{\norm}[1]{\Vert#1\Vert} 

\def\osc{\mathop{{\mathrm{osc}}}\limits}

\def\blbox{\quad \vrule height7.5pt width4.17pt depth0pt}

\newcommand{\beq}{\begin{equation}}
\newcommand{\eeq}{\end{equation}}
\newcommand{\beqs}{\begin{equation*}}
\newcommand{\eeqs}{\end{equation*}}
\newcommand{\beqa}{\begin{equation}\begin{aligned}}
\newcommand{\eeqa}{\end{aligned}\end{equation}}
\newcommand{\beqas}{\begin{equation*}\begin{aligned}}
\newcommand{\eeqas}{\end{aligned}\end{equation*}}
\newcommand{\brmk}{\begin{rmk}}
\newcommand{\ermk}{\end{rmk}}
\newcommand{\partref}[1]{\hbox{(\csname @roman\endcsname{\ref{#1}})}}
\newcommand{\half}{\frac{1}{2}}
\newcommand{\thalf}{\tfrac{1}{2}}
\newcommand{\mhalf}{{-\half}}
\newcommand{\na}{\nabla}

\newcommand{\pt}{\partial_t}
\newcommand{\abs}[1]{\vert#1\vert} 
\newcommand{\babs}[1]{\left\vert#1\right\vert}
\newcommand{\dist}{\text{dist}}

\newcommand{\supp}{\text{supp}} 
\newcommand{\eps}{\varepsilon}
\newcommand{\Ede}{E_\partial}
\newcommand{\leqs}{\lesssim}

\newcommand{\Loj}{{\L}ojasiewicz }

\newcommand{\DD}{\mathbb{D}}
\newcommand{\tc}{\tau_{\C}}
\newcommand{\tg}{\tau_g}
\newcommand{\ftau}{f_\tau}
\newcommand{\ddt}{\tfrac{d}{dt}}
\newcommand{\gz}{{g_{S^2}}}
\newcommand{\dvgz}{d v_{g_{S^2}}}

\newcommand{\la}{\lambda}
\newcommand{\Area}{\text{Area}}
\title[Rigidity estimates for maps between spheres]{\sc Sharp quantitative rigidity results for maps from  $S^2$ to $S^2$ of general degree}
\author{Melanie Rupflin}
\date{\today}

\begin{document}

\maketitle

\begin{abstract}
As the energy of any map $v$ from $S^2$ to $S^2$ is at least $4\pi \abs{\deg(v)}$ with equality if and only if $v$ is a rational map one might ask whether maps with small energy defect $\de_v=E(v)-4\pi \abs{\deg(v)}$ are necessarily close to a rational map. While such a rigidity statement turns out to be false for maps of general degree, we will prove that any  map $v$ with small energy defect is essentially given by a collection of rational maps that describe the behaviour of $v$ at very different scales and that the corresponding distance is controlled by a quantitative rigidity estimate of the form $\dist^2\leq C \de_v(1+\abs{\log\de_v})$ which is indeed sharp. 
\end{abstract}

\section{Introduction}
A natural question that arises in the study of many variational problems is the rigidity or stability of minimisers, i.e.~the question whether knowing that the energy of an object (such as a function, map, set etc) is close to the minimal possible energy $E_{min}$ is sufficient to deduce that the object is close to a minimiser. 

It is of particular interest to understand this question not only at a qualitative but also at a quantitative level, and hence to ask whether it is possible to bound the distance of an object $v$ to the set of minimisers by a function $f(\de_v)$ of the energy defect $\de_v=E(v)-E_{min}$,  with $f(\de)\to 0$ as $\de\to 0$ and if so, what the optimal such estimate is.

Over the past two decades there has been significant interest and very substantial progress for such problems related to a large variety of important topics in mathematics, including the isoperimetric problem, see e.g. \cite{sharp-isoperimetric, isoperimetric-FMP}, geometric variational problems \cite{DeLellis-Mueller, Almost-Schur, Lamm-Nguyen, Ciraolo, Yamabe, Luckhaus},
 elasticity theory \cite{Friesecke, DeLellis-Sz, ellastic1, Dolzmann} and 
Sobolev inequalities, see for instance \cite{Fusco-Sobolev, Figalli-Maggi-Pratelli-Sobolev, Faber-Krahn, Figalli-recent, Sobolev-Maggi-Neumayer} and the references therein.


Here we consider this question for a very classical problem, namely the Dirichlet energy 
$$E(v)=\thalf \int_{S^2} \abs{\na v} ^2 \dvgz \text{ of maps } v:S^2\to S^2$$
with given degree from $S^2=\{y\in \R^3: \abs{y}=1\}$ to itself and the corresponding 
minimisers which are rational maps, i.e.~maps that are given by meromorphic functions in stereographic coordinates. Compared with many of the known rigidity estimates, this problem turns out to have two rather surprising features:
First of all rigidity (even at the non-quantitative level) fails for maps of general degree if one considers the distance to the set of rational maps, but quantitative rigidity estimates are valid if one extends the set of comparison objects to collections of rational maps which concentrate on essentially disjoint domains. Secondly, the sharp rate in the quantitative rigidity estimates depends on the degree of the map and while the rigidity estimate \eqref{est:quant-deg-1}  for degrees $\pm 1$ maps that was obtained in \cite{Bernard-Muratov-Simon} has the well known form $\dist^2\leq C \de_v$ we  will see that for any other degree the sharp estimate takes the rather surprising form of $\dist^2\leq C \de_v\abs{\log\de_v}$.

To make this precise we 
recall that the Dirichlet energy 
of any map $v:S^2\to S^2$ is bounded below by 
$E(v)\geq 4\pi\abs{\deg(v)}$
with equality if and only if $v$ is a critical point of $E$, i.e.~a harmonic map. We furthermore recall the well known fact that a map from the standard sphere $S^2$ to itself is harmonic if and only if it is a rational map, i.e.~given by a meromorphic function  from $\hat \C$ to $\hat \C$ in either $z$ or $\bar z$ if we work in stereographic coordinates on both the domain and the target. 
It is hence natural to ask whether the distance of a map to the set of rational maps can be controlled in terms of the energy defect 
\beq
\label{def:energy-defect}
\de_v:= E(v)-4\pi \abs{\deg v}.
\eeq
This is of course trivially true for degree $0$ maps. For maps of degree $\pm 1$ Bernand-Mantel, Muratov and Simon gave a positive answer to this question in \cite{Bernard-Muratov-Simon} 
where they established  a sharp quantitative rigidity estimate, the proof  of which was subsequently simplified  by Topping in \cite{Topping-deg-1} and by Hirsch and Zemas in \cite{Hirsch} using ideas from geometric analysis. These papers establish
that for any degree $\pm 1$ map 
$v:S^2\to S^2$ there exists a degree $\pm 1$ 
rational map $\om:S^2\to S^2$, i.e.~a M\"obius transform in $z$ or $\bar z$, so that 
\beq
\int_{S^2} \abs{\na(v-\om)}^2 \dvgz \leq C\de_v.
\label{est:quant-deg-1} 
\eeq
In contrast, for  maps with $\abs{\deg(v)}=k\geq 2$ it is not only the above specific rigidity estimate that fails, as was already observed in \cite{Deng}, but rigidity fails even at a non-quantitative level. 

Namely,  we can construct a sequence of 
 maps $v_n$ of any given degree $k\geq 2$ so that 
\beq
\label{claim:intro} 
\de_{v_n}\to 0 \text{ but } \int_{S^2}\abs{\na( v_n-\om_n)}^2\geq c>0\text{ for all } n \text{ and all rational maps } \om_n, \eeq 
 e.g.~by gluing highly concentrated copies $\si(\mu_n z)$, $\mu_n\to \infty$, of a rational map $\si:\C \to S^2$
onto another fixed rational map $\si_0:S^2\to S^2$ at a point $p$ with $\si_0(p)\neq \si(\infty)$, see Section 
 \ref{sect:sharp}, and in particular Remark \ref{rmk:no-rigidity}, for details. 
While the 
 distance of these maps $v_n$ to any single rational map is of order $1$, 
they will essentially look like a combination of two rational maps which appear at increasingly different scales. 

The natural question to ask is hence not whether maps $v$ with small energy defect are close to a single rational map, but rather whether they are essentially described by a \textit{collection of rational maps}, which correspond to the behaviour of $v$ at \textit{very different scales}. We will indeed see that this is the case and that the corresponding distance is governed by a sharp quantitative rigidity estimate.

As the problem is invariant with respect to pull-back by M\"obius transforms we want to measure distances in ways that respect this symmetry. One natural choice is the 
$\dot H^1$-distance which is conformally invariant for maps from 2-dimensional domains and as our first main result we prove

\begin{thm}
\label{thm:3}
For any $\alpha<\infty$ and any $k\in \N$ there exists a constant $C$ so that for any map
 $v\in H^1(S^2,S^2)$ of degree $k$ there exists a 
collection of rational maps $\om_1,\ldots, \om_n$ from $S^2$ to $S^2$ with  $\deg(\om_i)\geq 1$ and  $\sum_{i=1}^n \deg(\om_i)= k$ 
and a corresponding partition of $S^2$ into disjoint subsets $\Om_i$,  obtained from balls $B_i$ by cutting out a (potentially empty) collection of smaller balls, so that the following holds: 

 The map 
$v$ is essentially given by $\om_i$
on $\Om_i$  in the sense that 
\beq
\label{claim:thm3-1}
\int_{\Om_i}\abs{\na(v-\om_i)}^2 \dvgz\leq C \de_v(\abs{\log(\de_v)}+1) \text{ for every } i
\eeq
and $\om_i$ is essentially constant outside of $\Om_i$ in the sense that 
\beq
\label{claim:thm3-2}
\int_{S^2\setminus \Om_i} \abs{\na \om_i}^2\dvgz\leq C (\de_v)^{2\al} 
 \text{ and } \osc_{U}\om_i \leq C (\de_v)^{\al}
\eeq
for every connected component $U$ of  $S^2\setminus \Om_i$ and every $i=1,\ldots,n$. 
\end{thm}

We note that $\al$ should be thought of as a very large number that can be freely chosen depending on what level of interaction between the different scales one wants to allow.

For maps $v$ with negative degree the analogue of Theorem \ref{thm:3}, and of all further results, of course also holds true and follows by composing such maps with a reflection.

We  will see below that 
these 
 domains $\Om_i$ correspond to vastly different scales in the sense that 
if we look at our problem in a gauge which is so that 
at least some of the energy of $\om_i$ appears at scale one (as opposed to essentially all of it concentrating on very small balls), then all other sets $\Om_j$ will correspond to sets with very small diameter. 
In such a view point $v$ is hence essentially described by the corresponding $\om_i$ while the other maps $\om_j$ look like they are constant.

We can use this idea of viewing the collection of rational maps $\om_i$ as representatives of $v$ in different gauges to formulate an alternative version of the above rigidity result.

 \begin{thm}
 \label{thm:gauge}
For any $\alpha<\infty$ and any $k\in \N$ there exists a constant $C$ so that for any map
 $v\in H^1(S^2,S^2)$ of degree $k\in \N$ there exists a 
collection of rational maps $\om_1,\ldots, \om_n$ from $S^2$ to $S^2$ with $\deg(\om_i)\geq 1$ and $\sum_{i=1}^n \deg(\om_i)=k$ 
 so that the following holds: 

For each $\om_i$ there is a collection of $m_i\leq \deg(\om_i)$ 
 M\"obius transforms $(M_{i}^j)_{j=1}^{m_i}$
 so that 
 \beq
\label{claim:thm-gauge-1}
\abs{\na \om_i}^2\leq C \rho_i^2 \text{ for } \rho_i^2:=\thalf \sum_{j=1}^{m_i}\abs{\na M_i^j}^2
\eeq
and so that $v$ is $L^2$ close to $\om_i$ in any of the corresponding gauges in the sense that 
\beq
\label{claim:thm-gauge-2}
\norm{v-\om_i}_{L^2(S^2,\rho_i^2 \gz)}^2=\sum_{j=1}^{m_i} \int_{S^2}\abs{(v-\om_i)\circ (M_i^j)^{-1}}^2 \dvgz 
\leq C \de_v (1+\abs{\log\de_v}).
\eeq
Furthermore, these $\om_i$ represent the behaviour of $v$ at vastly different scales in the sense that we can partition $S^2$ into connected sets $\Om_i$, obtained from balls in $S^2$ by removing a (possibly empty) set of smaller balls, in a way that the sets $\Om_i$ represent essentially all $S^2$ in the corresponding gauges in the sense that 
\beq
\label{claim:Omi-gauge}
\Area_{\rho_i^2 \gz}( S^2\setminus \Om_i)=
\sum_{j=1}^{m_i} \Area_{(M_i^j)^*\gz}(\Om_i)=
\sum_{j=1}^{m_i} \Area_{\gz}(M_i^j(S^2\setminus \Om_i))\leq C\de_v^{2\al}
\eeq
and
so that the oscillation of $\om_i$ over any connected component of $S^2\setminus \Om_i$ is bounded by $C\de_v^\al$. 
\end{thm}

We note that 
\eqref{claim:Omi-gauge} ensures that in any of the other gauges $\Om_i$ corresponds to a set with Area no more than $C\de_{v}^{2\al}$ and that \eqref{claim:thm-gauge-1} and \eqref{claim:Omi-gauge} immediately imply that the energy of $\om_i$ is essentially concentrated on $\Om_i$ as described in the second claim \eqref{claim:thm3-2} of the previous Theorem \ref{thm:3}. We will see that the above two results hold simultaneously for the same collection of rational maps and domains and will indeed prove Theorem \ref{thm:3} based on Theorem \ref{thm:gauge}. 

As an immediate consequence of Theorem \ref{thm:gauge} we also obtain the following

\begin{cor}\label{cor}
For any $\alpha<\infty$ and any $k\in \N$ there exists a constant $C$ so that for any map
 $v\in H^1(S^2,S^2)$ of degree $k\in \N$ there exists a 
collection of rational maps $\om_1,\ldots, \om_n$ from $S^2$ to $S^2$  with $\deg(\om_i)\geq 1$ and $\sum_{i=1}^n \deg(\om_i)=k$ so that
\beq
\label{claim:thm1}
\int_{S^2} \abs{\na \om_i}^2 \abs{\om_i-v}^2 \dvgz \leq C \de_v(1+\abs{\log\de_{v}}), \quad i=1,\ldots, n\eeq
and so that for each $i\neq j$ there exists a constant $c_{i,j}\in S^2$ with
$$\int\abs{\na \om_j}^2\abs{\om_i-c_{i,j}}\leq C (\de_v)^{2\al}.$$
\end{cor}

These results are sharp since we can also prove

\begin{thm}\label{thm:2}
For any $k\geq 2$ there exists  a constant $c_0>0$ and a sequence $v_n$ of maps with degree $k$ for which $\de_{v_n}\to 0$ so that
\beq\label{claim:thm2}
\int_{S^2} \abs{\na \om}^2 \abs{\om-v_n}^2 \geq c_0 \de_{v_n} (1+\abs{\log\de_{v_n}})
\eeq 
for any $n\in \N$ and for any rational map $\om_n$ with $1\leq \deg(\om_n)\leq k$, and so that
\beq\label{claim:thm2-2}
\sum_i\int_{\Om_i^n} \abs{\na (\om_i^n-v_n)}^2  \geq c_0 \de_{v_n} (1+\abs{\log\de_{v_n}})
\eeq
for 
 any partition of $S^2$ into $j\geq 1$ sets $\Om_1^n,\ldots,\Om_j^n$ and any collection of 
rational maps $\om_1^n,\ldots,\om_j^n$ with $\deg(\om_i^n)\geq 1$ and $\sum \deg(\om_i^n)=k$. 
\end{thm}

As Corollary \ref{cor} is an immediate consequence of Theorem \ref{thm:gauge}, we hence deduce 
that the rigidity estimates obtained in all of the above results, that is Theorems \ref{thm:3} and \ref{thm:gauge} as well as Corollary \ref{cor}, are sharp. 

We will obtain such maps $v_n$ by carefully gluing a highly concentrated bubble $\si_a(\pi(\mu z))$, $\pi:\C\to S^2$ the inverse stereographic projection, onto a base map $\si_0$ at a point $p$. We will choose $\si_a$ and $\si_0$ so that the distance between the images $\si_0(p)$ and $\si_a(\pi(\infty))$ at the points where we glue is  $a\sim  \mu^{-1}$, i.e.~so that this distance in the target is of the same order as the radius of the ball in the domain on which the bubble is essentially concentrated. Such maps will have distance of order $a\sim \mu^{-1}$ from any collection of rational maps, but can be constructed in a way that their energy defect tends to zero at a rate of 
$O(a^2\log(\mu)^{-1})=O(\mu^{-2}(\log\mu)^{-1})$ as $\mu\to \infty$, 
 see Section \ref{sect:sharp} for details.

Outline of the paper: \\
We will prove Theorems \ref{thm:3} and \ref{thm:gauge} by combining a \Loj estimate with a rather delicate flow argument in which we evolve the given map with a weighted flow that uses weighted domain metrics  that are adapted to the energy distribution of the given map $v$. We give an overview of this argument in Section \ref{sec:overview} and then carry out the detailed analysis in the subsequent Section \ref{sec:details}. This argument uses a new  \Loj estimate that involves such weighted domain metrics which is stated in Proposition \ref{prop:1} and is proven later on in Section \ref{sec:Prop}, and we will prove Theorem \ref{thm:2} in Section \ref{sect:sharp}.

\textbf{Acknowledgement:} The author would like to thank Peter Topping for interesting conversations on this topic. 

\section{Overview of the proofs of Theorems \ref{thm:3} and \ref{thm:gauge} }
\label{sec:overview}
In the present section we give an overview of the key arguments that lead to the proof of the rigidity estimates claimed in our main results.  

Since the claims of these results are trivially true for maps whose energy defect is bounded away from zero it suffices to consider maps with 
\beq
\label{ass:De_E_small}
\de_v\leq \bar\de \text{ for some } \bar \de=\bar\de(k,\al)\in (0,\thalf) ,
\eeq
$k\in \N$ and $\al<\infty$ the fixed numbers for which we want to establish our main results.
In the following we hence use the convention that all arguments and results are to be understood to hold for maps $v$ satisfying \eqref{ass:De_E_small} for a sufficiently small such $\bar \de=\bar \de(k,\al)$.

To explain the basic idea of our proof we first recall 
that 
the $L^2(S^2,g)$ gradient of the Dirichlet energy of a map $u:S^2\to S^2\hookrightarrow \R^3$ with respect to a domain metric $g$ is described by the so called tension field which can be computed as 
$$\tau_g(u)=\Delta_g u+\abs{\na_g u}_{g}^2 u$$
if we view $u$ as a map into the surrounding Euclidean space. 
We stress that unlike the energy this quantity and the corresponding $L^2$ norm are not conformally invariant but rather scale according to 
 $\tau_g(u)= \rho^{-2}\tau_{\gz}(u)$
 and $\norm{\tau_g(u)}_{L^2(S^2,g)}=\norm{\rho^{-1}\tau_\gz}_{L^2(S^2,\gz)}$ for $g=\rho^2 \gz$. 

We recall that
 \Loj estimates for maps $v:S^2\to S^2$ of the form  
\beq
\label{est:Loj-standard}
\de_v\leq C \norm{\tau_{g_{S^2}}(v)}^2_{L^2(S^2,\gz)}
\eeq
 that involve the tension with respect to the standard metric $g_{S^2}$ on $S^2$
were established in \cite{Topping-rigidity} by Topping, who then also used these estimate in \cite{Topping-deg-1} to give a
proof of the quantitative estimates \eqref{est:quant-deg-1} for degree $1$ maps. 

The basic idea of the proof in \cite{Topping-deg-1} is to first pull back a given degree $1$ map $v$ 
by a M\"obius transform to make it balanced in the sense that $\int_{S^2} v \dvgz=0$, and to then use \eqref{est:Loj-standard} to show that the corresponding solution of 
the classical harmonic map flow 
 \beq
\label{def:HMF}
\partial_tu=\tau_{g_{S^2}}(u), \quad u(0)=v \eeq
 remains smooth for all times and converges smoothly to a limiting harmonic map $u_\infty$ with uniformly bounded energy density. A key point of that proof is that for degree $1$ maps with small $\de_v$ it is impossible that energy concentrates at multiple points or at multiple scales, so working in a gauge in which the map is balanced excludes the possibility that energy is concentrated on any small ball.

This is no longer the case for maps of degree $k\geq 2$. While we would like to follow a similar strategy in that we want to turn \Loj estimates into rigidity estimates by using a gradient flow, we hence cannot expect the proof from the degree $1$ case \cite{Topping-deg-1} to work as energy can be concentrated at multiple points and at multiple scales. 

For general maps a balancing condition as considered in \cite{Topping-deg-1} is hence neither sufficient to exclude the formation of singularities of the standard harmonic map flow nor to gain sufficient control on the energy density of the limiting harmonic map of that flow. Indeed, as observed above,  
we cannot expect such maps $v$ to be close to any single rational map $\om$ even if their energy defect is very small, so  cannot hope to deform $v$ to  such a limit using any single flow, unless we know a priori that the energy of $v$ is not concentrating at very different scales, compare Remark \ref{rmk:simple-case} below.

Instead we will associate to each given map $v$ a collection of weighted domain metrics $g_i=\rho_i^2 g_{S^2}$ on $S^2$ where the chosen weights $\rho_i^2$  reflect the distribution of the energy of $u_0$, with different metrics corresponding to different regions of the domain on which energy concentrates.  
For each such metric we  then consider the weighted harmonic map flow 
\beq
\label{def:HMF-w-g}
\pt u=\tau_{g_i}(u)=\rho_i^{-2} \tau_{g_{S^2}}(u)
\eeq
with initial map $u_0=v$. 
These weighted metrics will be obtained by combining the pull-backs of the standard metric $g_{S^2}$ by  M\"obius transforms which rescale certain regions on which the energy of $u_0$ is highly concentrated to unit size. 

Using a weighted flow that reflects the properties of the initial map $u_0$ turns out to have several major advantages: 
On the one hand we will control the velocity $\pt u$ in a weighted $L^2$ sense and hence obtain much stronger bounds on the distance  of $u_0$ to the resulting limiting harmonic map $u_\infty$ on regions where the weight is large, i.e.~on regions where $u_0$ was highly concentrated.
Working with a weighted metric will also give us improved control on the formation of singularities of the flow and, just as importantly, will allow us to prove the bound \eqref{claim:thm-gauge-1} 
on the energy density of the limiting harmonic map claimed in Theorem \ref{thm:gauge}, which in turn will be crucial also in the proof of Theorem \ref{thm:3}.

To be able to carry out this approach we however need \Loj estimates that involve the tension with respect to such weighted metrics. Such estimates do not follow from \eqref{est:Loj-standard} as 
we will need to work with respect to metrics $g=\rho^2\gz$ for which the conformal factor will be very large on the small 
regions that we rescale and for which
$ \norm{\tau_g(u)}_{L^2(S^2,g)}=\norm{\rho^{-1}\tau_\gz(u)}_{L^2(S^2,\gz)}$ can thus be far smaller than the norm of the standard tension. 

To prove our main results we hence need to establish a new \Loj estimate and in Section  \ref{sec:Prop} we will  prove the following proposition that might be of independent interest.

\begin{prop}
\label{prop:1}
For any $d\in \N$ there exist constants $C>0$ and $\eps>0$ so that the following holds true 
for any  collections of no more than $d$ points $p_i\in S^2$ and radii $r_i \in (0,\frac{\pi}2)$. 
Let $M_{p_i,r_i}:S^2\to S^2$ be the M\"obiustransform that corresponds to a dilation centred at $p_i$ that scales $B_{r_i}(p_i)$ up to a hemisphere, let $\rho_{p_i,r_i}=\frac{1}{\sqrt2} \abs{\na M_{p_i,r_i}} $ be the corresponding conformal factors and let 
\beq \label{def:g}
g:= \gz+\sum_i M_{p_i,r_i}^*g_{S^2}=(1+\sum_i \rho_{p_i,r_i}^2) \gz.
\eeq
Then 
for any map $u:S^2\to S^2$ with $\abs{\deg(u)}\leq d$
for which 
\beq
\label{ass:Loj-est}
\de_{u} \prod_i(1+\abs{\log r_i})\leq \eps\eeq
we have
\beq \label{est:Loj-weighted}
\de_{u}\leq C(1+\max\abs{\log r_i}) \norm{\tau_g(u)}_{L^2(S^2,g)}^2.
\eeq
\end{prop}

Here and in the following we denote by $B_r(x)=\{y\in S^2: \dist_{\gz}(x,y)<r\}\subset S^2$ the geodesic ball with respect to the standard metric $\gz$ on the sphere $S^2\hookrightarrow \R^3$. 

We want to use these \Loj estimates to show that we can deform the given map $v$ into a harmonic maps whose distance from $v$ is of order $O(\de_v ^\half \abs{\log\de_v}^\half)$ using a suitable weighted flow. To this end we will need to ensure that the potentially very large factor $(1+\max \abs{\log r_i})$ appearing in \eqref{est:Loj-weighted} is always controlled by $\abs{\log\de_v}$. We will hence only 
ever work with domain metrics which, when viewed in the right gauge  i.e.~pulled back by a (single) suitable M\"obius transform, can be written in the form 
\eqref{def:g} for radii satisfying
\beq
\label{ass:mui}
\abs{\log r_i}\leq C_1\abs{\log\de_{v}} \text{ for some fixed } C_1=C_1(k,\al).
\eeq

\begin{rmk}\label{rmk:simple-case}
Our argument simplifies significantly if 
 the scales at which energy of $v$ concentrates are not too different in the sense that 
there is a single M\"obius transform $M$ so that \eqref{ass:mui}
is satisfied for the radius of any ball which  contains a certain amount $\eps_1=\eps_1(k)>0$ of energy of
 $v\circ M^{-1}$. In this case we can simply scale up a suitable collection of such balls to unit size and evolve by the corresponding weighted flow in order to show that $v$ will be close to a single rational map $\om$. Of course, an even simpler case is if all these balls have radius of order one, which is e.g.~the case if one focuses on maps that are in a small neighbourhood of a compact set of rational maps in $H^1$ as done in \cite{Deng}, and in that special case already Topping's approach of evolving by the standard harmonic map flow carried out in \cite{Topping-deg-1} is sufficient to obtain \eqref{est:quant-deg-1}.
 \end{rmk}
For general maps we have to proceed with a lot more care since  the energy of $v$ can be concentrated at vastly different scales, and since we should not expect to find a single harmonic map $\om$ that describes $v$ at all these scales. Instead we want to determine a suitable collection of harmonic maps $\om_i$, each of which will capture the behaviour of $v$ at very different scales and on corresponding domains $\Om_i$, as described in our main results.

To extract each of the maps $\om_i$ will use 
the following three-step-procedure:

\textbf{Step 1:} 
 We first evolve $v$ with the weighted harmonic map flow for a domain metric $ g_i$ which is chosen so that the resulting limiting harmonic map $\tilde \om_i$, which will have  energy defect $\de_{\tilde \om_i}\leq \de_v$, is close to $v$ on the corresponding set $\Om_i$. 

 To achieve this we will work with a metric $g_i$ 
  which scales up all parts of  $\Om_i$ that contain a certain amount of energy to order $1$ and which involves some additional 
 weights that are highly concentrated outside of $\Om_i$ and that prevent energy from moving between $\Om_i$ and the other domains $\Om_j$, $j\neq i$, that will be captured by separate flows and that correspond to very different scales. 

\textbf{Step 2:}
While this first harmonic map  $\tilde \om_i$ will be well controlled on
$\Om_i$ 
it can still contain energy that is highly concentrated 
on $S^2\setminus \Om_i$ and that is unrelated to the behaviour of the initial map $v$ on $\Om_i$. In the second step we hence 
 cut out all of these highly concentrated parts to obtain a new map $\tilde v_i$ which is still close to the original map $v$ on $\Om_i$, still has very small energy defect, but now has very little energy outside of $\Om_i$. 
 
 \textbf{Step 3:}
This gives us a new initial map $\tilde v_i$ which is now so that  all regions that carry a certain amount of energy can be rescaled to unit size 
 without violating \eqref{ass:mui}. We can hence flow this 
 new map $\tilde v_i$ again, now with a flow that uses a metric $\tilde g_i\leq g_i$ which is essentially supported on $\Om_i$, 
  and prove that this second flow remains smooth for all times and  results in a limiting harmonic map $\om_i$ that has all of the required properties.

To ensure that we can carry out these arguments on suitable disjoint domains   $\Om_i$ and that we can capture the behaviour of $v$ at all relevant scales, we will extract these maps in the following order: 

We first pull back the given map $v$ to a map 
 $v_1$ for which at least some of the energy  appears at scale $1$ (rather than essentially all of it concentrating on very small balls). 

We then identify a collection of no more than $k$ balls $B_{s_i}(y_i)$, $i\in I_0=\{1,\ldots, J\}$, which describe the location and scales at which the energy of $v_1$ concentrates. We will see that such  balls can be chosen 
so that 
most of the energy of $v_1$ is contained in the union of balls $B_{\La s_i}(x_i)$ for some fixed $\La=\La(k)$, and so that each of these regions contributes  energy of about $4\pi k_i$ for numbers $k_i\geq 1$ with $\sum k_i=k$. 

In a first step we then extract a harmonic map that represents 
the behaviour of $v$ at scales which are not too small in this view point.
 To this end we distinguish between balls  whose radii  are bounded below by some large (but fixed) power of $\de_v$ and that can hence be captured in the current step and balls that are too small to be rescaled. 
 
This distinction will correspond to partitioning the set of all indices $I_0$ into a non-empty set $I_1\subset I_0$ and a (potentially empty) set $I_*=I_0\setminus I_1$ so that there are
$s_1^*$ and $S_1^*$ with
$$\abs{\log(s_1^*)}\leq C_1 \abs{\log\de_v} \text{ and } 
\max\{ s_i: i\in I_*\}\ll  s_1^*\ll S_1^*\ll \min\{ s_i: i\in I_1\},$$
where  $a \ll b$ means $a\leq C \de_v^{\al_1}b$ for some $C=C(k,\al)$ and a large $\al_1$. 

We will then carry out Step 1 using a metric $g_1$ that scales up all balls that correspond to indices in  $I_1$ to order one and that includes additional weights that correspond to balls around points $y_i$ for $i\in I_*$. These additional weights are obtained by rescaling balls  with the fixed radius $s_1^*$ rather than $s_i$ to ensure that \eqref{ass:mui} holds.  

Evolving with this flow then yields our first harmonic map $\tilde \om_1$, which might however be highly concentrated on the balls $B_{s_1^*}(y_i)$, $i\in I_*$. 

We will then group these highly concentrated balls into \textit{clusters} by partitioning $I_*$ into subsets $I_{*}^1,\ldots, I_*^{m_1}$  so that the collections of balls  $B_{S_1^*}(y_i)$, $i\in I_*^j$, with radius $S_1^*\gg s_1^*$ can be covered by pairwise disjoint balls $B_*^j$, $j=1,\ldots,m_1$, whose radii are of order $S_1^*$. 
 
We then define our first domain $\Om_1$ as 
 $\Om_1:=S^2\setminus \bigcup_{j=1}^{m_1} B_*^j$ 
 and cut out the highly concentrated parts of the map by 
 replacing $\tilde \om_1$ on each of these disjoint balls $ B_*^j $ with a map that has the same boundary values but very little energy. This then allows us 
to obtain the desired $\om_1$ by flowing again, as described in Step 3 above. 

Having hence extracted the first harmonic map $\om_1$ we then want to repeat this argument to extract all of the information that we have lost on the above clusters of balls $B_{s_i}(y_i)$, $i\in I_*$, on which $v_1$ is highly concentrated.

As the balls $B_*^a$ are disjoint, and as the radii of these balls are very large compared to the original radii $s_i\ll S_1^*$, $i\in I_*$,
we will be able to carry out this argument separately for each cluster. Indeed, we will see that if we rescale any of these very small balls $B_{s_i}(y_i)$ to order $1$ then all parts of the domain that we have either already captured or that correspond to another cluster will be scaled down to balls which are so small that they are guaranteed to be cut out when we repeat the above 3-step-argument in this new gauge. 
 
 For each 
 $B_*^a$, $a\in \{1,\ldots,m_1\}$, we can thus extract the
next \textit{layer of harmonic maps} $\om_i^a$, which will correspond to disjoint 
index sets $I_i^{a}\subset I_*^a$ and 
disjoint domains $\Om_i^a\subset B_*^a$. These new domains $\Om_i^a$ will in turn be 
obtained from subsets of $B_*^a$ 
from which a further (potentially empty) set of even smaller balls $B_{**,i}^{a,b}$ is cut out. These new balls $B_{**,i}^{a,b}$ then correspond to a  
second level of clusters of highly concentrated regions from which we continue to extract harmonic maps using the same arguments. 

This procedure will not only ensure that the resulting subsets $\Om_\beta$ of $S^2$ and $I_\beta$ of $I_0$ partition $S^2$ and $I_0$ into disjoint sets, but will also allow us to prove that these sets $\Om_\beta$ and the extracted harmonic maps $\om_\beta$ all have the following key properties: 
\begin{enumerate}[label=(K\arabic*)]
\item 
\label{key:1} 
The degree of $\om_\beta$ is given by $\deg(\om_\beta)=\sum_{i\in I_\beta} k_i\geq 1$.
\item 
\label{key:2}
The energy density of $\om_\beta$ is controlled pointwise by 
$
\abs{\na \om_\beta}^2\leq C \rho_{\beta}^2 
$
for a conformal factor $\rho_{\beta}$ that is obtained by rescaling 
a collection of no more than $\deg(\om_\beta)$ balls in $\Om_{I_\beta}$ to hemispheres.
\item 
\label{key:3}
We can bound 
$\int_{\Om_\beta} \rho_{\beta}^2 \abs{\om_\beta-v}^2 \dvgz\leq C \de_v \abs{\log\de_v}.$
 \item 
 \label{key:4} The oscillation of $ \om_\beta$ over each connected component of  $S^2\setminus \Om_\beta$ is bounded by $C\de_v^\al$ and we have
   $\int_{S^2\setminus \Om_\beta}\rho_\beta^2 \dvgz\leq C  \de_{v}^{2\al}$.
 \end{enumerate}

Property \ref{key:1} and the fact that $\sum_{I_0} k_i=k$ ensure that once 
 we have captured all of the balls $B_{s_i}(y_i)$, $i\in I_0$, we indeed have a collection of rational maps with total degree $k$. The other key properties \ref{key:2}-\ref{key:4} then establish all of the claims made in our second main result Theorem \ref{thm:gauge}. 

The details of all of these arguments are carried out in 
 the following Section \ref{sec:details} which starts with a description of properties of the weighted harmonic map flow and of rational maps,  see Sections \ref{subsec:flows} and \ref{subsec:rational-main}. We  
then explain in Section 
 \ref{subsec:balls} how to choose the required collection of balls $B_{s_i}(y_i)$, before carrying out all of the  details of the construction of the first harmonic map $\om_1$ and of the proofs of the above key properties \ref{key:1}-\ref{key:4} for this map in Section \ref{subseq:first-step}. In the subsequent Section \ref{subsec:next-step} we then explain how these arguments can be modified to extract all further harmonic maps $\om_i$ and to obtain the required properties \ref{key:1}-\ref{key:4} for these $\om_i$. 
At this stage we will then have completed the proof of Theorem \ref{thm:gauge} and we finally explain how this result implies our first main result Theorem \ref{thm:3} in Section \ref{subsec:last-step}.

\section{Detailed construction and analysis of the harmonic maps $\om_i$}
\label{sec:details}
\subsection{Evolving maps by weighted harmonic map flows}
\label{subsec:flows} $ $\\
In the following we will evolve our given map $v$, respectively suitable rescalings and modifications $u_0$ of $v$, by weighted harmonic map flows. 

To this end we first recall that the results of Struwe \cite{Struwe} ensure that for any
metric $g$ on $S^2$ and any initial map $u_0\in H^1(S^2,S^2)$ there exists a global weak solution $u$ of \eqref{def:HMF-w-g} which is smooth away from finitely many points in space-time where harmonic maps  bubble off. We also recall that along this solution the energy decays according to $\ddt E(u(t))=-\norm{\tau_g(u)}_{L^2(S^2,g)}^2$ away from the singular times 
where the energy drops by $\sum E(\si^j)$, $\si^j$ the harmonic maps that bubble off along a sequence $t_i\upto T$, compare \cite{DT, QingTian, LinWang}. 

We will only ever consider initial maps with positive degree and note that if the degrees $\deg(\si^j)$ of all bubbles that develop at a singular time are also positive then the energy defect remains continuous across the singular time since the change in the contribution of $4\pi \deg(u)$ caused by the change in degree of the map on the one hand and the loss of the energy on the other hand cancel out.

\newcommand{\sign}{\text{sign}}
Conversely, if we had any bubble with negative degree then $\de_{u(t)}$ would need to drop by at least $8\pi$, which in our arguments is impossible as we evolve maps with small energy defect.

In the following we will always deal with flows which converge in $L^2(S^2,g)$ to a (unique) limit $u(\infty)$ as $t\to\infty$ and the above argument also ensures that $\de_{u(\infty)}=\lim_{t\to \infty}\de_{u(t)}$. 

The energy defect hence decreases along the flow according to 
\beq
\label{eq:ddt-energy-defect}
\de_{u(t)}=\de_{u(0)}-\int_{0}^{t}\norm{\tau_g(u)}_{L^2(S^2,g)}^2 \text{ for all } t\in [0,\infty]. \eeq
As in \cite{Simon} we now want to  combine the evolution equation for the energy defect with a \Loj estimate to obtain an upper bound on the total distance that the flow travels.

To be able to apply Proposition \ref{prop:1} and to obtain a suitable bound for the right hand side of \eqref{est:Loj-weighted} we only ever want to consider weighted harmonic map flows for metrics $g$ as in \eqref{def:g} for which \eqref{ass:mui} is satisfied. 
As we only need to consider maps with small energy defect this also ensures that \eqref{ass:Loj-est} is satisfied along the flow. Hence 
 Proposition \ref{prop:1} yields
 \beq 
 \de_{u(t)}\leq C \abs{\log(\de_{u_0})} \norm{\tau_g(u(t))}_{L^2(S^2,g)}^2 \text{ for all } t.
 \eeq
We conclude that
 \beqs-\tfrac{d}{dt} (\de_{u(t)})^\half =\thalf (\de_{u(t)})^{\mhalf }\norm{\tau_{g}(u(t))}_{L^2(S^2,g)}^2\geq c \norm{\tau_{g}(u)}_{L^2(S^2,g)} \abs{\log \de_{u_0}}^{-\half} \eeqs
 for some $c=c(k,\al)>0$. We thus obtain
 that any such flow converges in $L^2(S^2,g)$ 
 as $t\to \infty$ to a (unique) limiting harmonic map $u(\infty)$ for which 
\beq 
\label{est:distance-limit-flow}
\norm{u_0-u(\infty)}_{L^2(S^2,g)}\leq \int_0^\infty \norm{\tau_{g}(u)}_{L^2(S^2,g)}\leq C \de_{u_0}^\half \abs{\log(\de_{u_0})}^\half=:C\xi_{u_0}.
\eeq
where here and in the following it is useful to abbreviate 
\beq
\label{def:xi}
\xi_u:=(\de_{u})^{\half} \abs{\log(\de_{u})}^\half.
\eeq
We also remark that the $L^2$ convergence of the flow, combined with the $C^k$ estimates that the maps $u(t)$ satisfy away from the singular set, compare \cite{Struwe}, ensures that $u(t)$ converges smoothly locally to $u(\infty)$ away from the points at which the flow becomes singular as $t\to \infty$.

In addition to obtaining a bound on the distance between the initial map and the limiting harmonic map we will also use \eqref{est:distance-limit-flow} to control the evolution of local energies. 

Given a cut-off function $\phi\in C^\infty(S^2,[0,1])$ we let
\beq \label{def:cut off-energy} 
E_\phi(t):= \half \int_{S^2} \phi^2 \abs{d u}_g^2 dv_g
\eeq
and use that such cut-off energies evolve according to 
\beqa
\ddt E_\phi(t)= \norm{\phi\tau_g(u)}_{L^2(S^2,g)}^2
+\int_{S^2} \phi \langle d \phi, d u\rangle_g \tau_g(u) dv_g
\eeqa
away from the finite singular set, compare also \cite{Topping-winding}.
From \eqref{eq:ddt-energy-defect} and \eqref{est:distance-limit-flow} we hence immediately obtain that if $[t_1,t_2]\times \supp(\phi)$ is disjoint from the singular set then 
\beqa \label{est:energy-cut}
\abs{E_\phi(t_2)-E_{\phi}(t_1)}&\leq \de_{u_0}+C \norm{d \phi}_{L^\infty(S^2,g)} \de_{u_0}^\half \abs{\log(\de_{u_0})}^\half
\\
&\leq C(1+\norm{d \phi}_{L^\infty(S^2,g)})\xi_{u_0}.
\eeqa 
We note that if $\supp(\phi)$ contains no point at which the flow becomes singular as $t\to \infty$ then this estimate also holds for 
 $t_2=\infty$, i.e.~for the cut-off energy $E_\phi(\infty)$ of the limiting harmonic map $u(\infty)$. 

We note that the lower bound on $E_\phi(t_2)$ that one can obtain from \eqref{est:energy-cut} will not hold if a bubble develops on $[t_1,t_2]\times \supp(\phi)$ as this will cause a drop in energy. Conversely the corresponding upper bound on $E_\phi(t_2)$  remains valid so we can always use that
\beq
\label{est:energy-cut-upper} 
E_\phi(t)\leq E_\phi(0)+ C(1+\norm{d \phi}_{L^\infty(S^2,g)}) \xi_{u_0}
\eeq
both 
for all finite $t$ and, by the weak lower semicontinuity of the energy, also for $t=\infty$. 

We will apply these estimates for cut-off functions $\phi$ that are constructed from 
 functions $(\varphi_{x,r})_{x\in S^2,r\in(0,\frac\pi2]}$  which are so that $\varphi_{x,r}\in C^\infty_0(B_{r}(x))$ satisfies $\varphi_{x,r}\equiv 1$ on $B_{r/2}(x)$ and $\norm{d\varphi_{x,r}}_{L^\infty(S^2,g_{S^2})}\leq C{r}^{-1}$ for some universal $C$. While $\norm{d\varphi_{x,r}}_{L^\infty(S^2,g_{S^2})}$ will be large for $r>0$ small,  we will crucially use that the corresponding quantity is well controlled if we work with respect to a suitable weighted metric. To be more precise, as
  $\rho_{\bar x,\bar r}$ 
  is of order $(\bar r)^{-1}$ at points whose distance from $\bar x$ is bounded by a multiple of $\bar r$, compare Lemma \ref{lemma:dilation}, we immediately get

\begin{lemma}
\label{lemma:cut offs}
For any numbers $\La\geq 1$ and any $d>0$ there exists a constant $C$ so
that the following holds true for any $\bar x\in S^2$ and any $\bar r\in (0,\pi/2]$. 

Let 
$\rho_{\bar x,\bar r}$ be the conformal factor of the M\"obius transform that corresponds to a dilation at $\bar x$ that scales $B_{\bar r}(x)$  up to a hemisphere. Then for any $x\in B_{\La \bar r}(\bar x) $ and $r\in (d \bar r,\bar r)$ we have 
\beq\label{est:cut off}
\norm{d \phi_{x,r}}_{L^\infty(S^2,\rho_{\bar x,\bar r}^2g_{S^2})}\leq C.
\eeq
\end{lemma}

This estimate of course also applies if we work with respect to a metric 
$g$ as in \eqref{def:g} that contains $\rho_{\bar x,\bar r}$ as one of the weights, since 
$\abs{d\phi}_{\rho^2 g_{S^2}}= \rho^{-1} \abs{d\phi}_{g_{S^2}}\leq \abs{d\phi}_{\tilde \rho^2 g_{S^2}}$ whenever $\rho\geq \tilde \rho$.

\subsection{Controlling the obtained harmonic maps}
\label{subsec:rational-main} $ $\\
In addition to these properties of the harmonic map flow 
 we 
will exploit that 
for maps $h:S^2\to S^2$ which are harmonic, and hence given by meromorphic functions in either $z$ or $\bar z$, 
we can turn 
estimates on the energy on suitable regions into pointwise control on the energy density. Namely, we will use the following lemma, a proof of which is included in the appendix

\begin{lemma}
\label{lemma:rational}
For any $k\in \N$ and $d>0$ there exist constants $\eps_1=\eps_1(k)>0$ and $C=C(k,d)$ so that the following holds true: 

Let $h:S^2\to S^2$ be any rational map with $\abs{\deg(v)}\leq k$ and suppose that we have a non-empty collection $(B_{R_i}(z_i))_{i\in J_1}$  and a further, possibly empty, collection
 $(B_{R_i}(z_i))_{i\in J_2}$
  of in total $\abs{J_1}+\abs{J_2}\leq k$ balls 
which contain most of the energy of $h$ in the sense that 
\beq
\label{ass:lemma-rat-1}
E(h,S^2\setminus \bigcup_{i\in J_1\cup J_2} B_{R_i}(z_i))\leq 2\eps_1
\eeq 
and which have the following properties: 
For each $i\in J_2$ we have 
\beq
\label{ass:lemma-rat-3}
E(h,B_{2R_i}(z_i)\setminus B_{R_i}(z_i))\leq 2\eps_1 
\eeq
while for each $x\in \bigcup_{i\in J_1} B_{4R_i}(z_i)$ there exists $j\in J_1$ so that
\beq\label{ass:lemma-rat-2} 
 x\in  B_{d^{-1} R_j}(z_j)\text{ and }
E(h,B_{d R_j}(x)\setminus \bigcup_{i\in J_2} B_{R_i}(z_i))\leq 2\eps_1 .
\eeq
Then the energy density of $h$ is controlled by 
\beq
\label{claim:lemma-rat}
\abs{\na h}^2\leq C\sum_{i\in J_1\cup J_2} \rho_{z_i,R_i}^2 \text{ on }  S^2\setminus \bigcup_{i\in J_2} B_{4R_i}(z_i).
\eeq
\end{lemma}
\begin{rmk}\label{rmk:rho-convention}
Here we use the convention that $\rho_{z,R}$ denotes the conformal factor of the M\"obius transform $M$ which is chosen as the dilation that scales the ball $B_R(z)$ up the corresponding hemisphere if $R<\pi/2$ while for $R\geq \pi/2$ we set $\rho_{z,R}=1$. 
\end{rmk}

On the one hand, we will apply this lemma in Step 2 of our 3-step-argument to show that we can 
cut out the highly concentrated regions of the harmonic map $\tilde \om_i$  without increasing the energy defect by more than $C\de_v^{\al_i}\leq  \de_v$. On the other hand, we will also use this lemma in Step 3 to prove that the energy density of the final harmonic map $\om_i$ indeed satisfies the crucial estimate \eqref{claim:thm-gauge-1}. 

In both of these arguments the first collection  of balls will correspond to the balls in the domain $\Om_i$ that we rescale. 
For the application in Step 2 the set $J_2$ 
will correspond to the additional weights we use to prevent energy from flowing between $\Om_i$ and its complement, while we will be able to choose $J_2$ to be the empty set when we apply this lemma in the third step and will hence obtain pointwise bounds on $\om_i$ that apply on all of $S^2$.

\subsection{Selection of a suitable collection of balls}
\label{subsec:balls} $ $\\
We first need to select suitable 
 regions of the domain on which a certain amount of energy concentrates.
To this end we fix $\eps_1=\eps_1(k)\in (0, \frac{\pi}{4(k+1)^2})$ so that 
 Lemma \ref{lemma:rational} holds, fix $\eps_2=\eps_2(k)>0$ so that Lemma \ref{lemma:balls} below applies and 
let  $\La_0:= 2^{\frac{4\pi k+1}{\eps_2}+3} $. We stress that these constants only depend on $k$ and hence that all constants $C$ that appear in the rest of this subsection also just depend on $k$.

Given a map $v$ with  degree $k\geq 2$ which satisfies \eqref{ass:De_E_small}  we then let $r_1>0$
 be so that  
 \beq
\sup_{x\in S^2} E(v,B_{r_1}(x))=\eps_1,\eeq
 choose $x_1\in S^2$ so that
 $E(v,B_{r_1}(x_1))=\eps_1$ and note that our choice of 
$\La_0$ ensures that we can always
pick a $\la_1\in [1, \La_0]$ so that 
\beq
\label{eq:la_1}
E(v,B_{2\la_1r_1}\setminus B_{\la_1r_1}(x_1))
\leq \eps_2.
\eeq
We then continue to 
 pick points $x_i$ and radii $r_i\geq r_{i-1}$ so that 
\beq
\label{eq:x_i}
E(v,B_{r_i}(x_i)\setminus \bigcup_{j\leq i-1} B_{\la_j r_j}(x_j))=\sup_{x\in S^2} E(v,B_{r_i}(x)\setminus \bigcup_{j\leq i-1} B_{\la_j r_j}(x_j))=\eps_1
\eeq
and factors  $\la_i\in [1,\La_0]$
so that 
\beq
\label{eq:la_i}
E(v,B_{2\la_ir_i}\setminus B_{\la_ir_i}(x_i))\leq \eps_2
\eeq until we get a collection of
such balls $B_{r_i}(x_i)$, $i\in I_0:=\{1,\ldots,J\}$ with
\beq
\label{est:rest-of-energy}
E(S^2\setminus \bigcup_{j\leq J} B_{\la_j r_j}(x_j))<\eps_1.\eeq 
Setting 
\beq
\label{def:F0-j} 
F_j^0:=\bigcup _{i\leq j} B_{\La_0r_i}(x_i) \text{ for  } j\in I_0 \text{ and } F_0^0=\emptyset \eeq
 we can hence 
 use in the following that  
\beq
\label{est:item0-2-1}
E(v,S^2\setminus F_{J}^0)\leq \eps_1
\eeq
 and   that
 \beq 
 \label{est:item0-2-2}
 E(v, B_{r_j}(x)\setminus F_{j-1}^0)\leq \eps_1 \text{ for every }x\in S^2 \text{ and } j=1,\ldots J.\eeq

While a priori the energy of $v$ on $B_{\la_i r_i}(x_i) \setminus \bigcup_{j\leq i-1} B_{\la_j r_j}(x_j)$ looks to be bounded below only by $\eps_1$, we indeed have
 \begin{lemma}
\label{lemma:balls}
There exists $\eps_2=\eps_2(k)>0$ so that if we carry out the above construction with this choice of $\eps_2$ then there exist positive numbers 
$k_i\in \N$ with $\sum_{i=1}^J k_i=k$
so that
\beq
\label{est:E-balls-4pi}
E(v,B_{\la_i r_i}(x_i) \setminus \bigcup_{j\leq i-1} B_{\la_j r_j}(x_j))\geq 4\pi k_i -\eps_1 \text{ for each } i\in I_0:= \{1,\ldots, J\}.
\eeq
In particular, $J\leq k$ and 
\beq
\label{est:E-balls-unit}
E(v,B_{\la_i r_i}(x_i) \setminus \bigcup_{j\leq i-1} B_{r_i}(x_j))\geq 4\pi k_i -i\eps_1 \text{ for each } i\in I_0.
\eeq
\end{lemma}

In the proof of this lemma, and also in  some later arguments, it can be convenient to exploit the conformal invariance of the energy to 
 view a map from  
a diadic annulus $B_{2R}\setminus B_R(x)$, $R<\pi/2$, on the sphere
instead as a map from the annulus 
 $\DD_{\tan(R)}\setminus \DD_{\tan(R/2)}$ in the plane, or, 
using the conformal map $re^{i\theta}\mapsto (\log r-\log(\tan R/2),\theta)$,
 as a map from the cylinder $[0,c_R]\times S^1$ for $c_R:= \log(\tan(R))-\log(\tan R/2) \geq \log2$.

\begin{proof}[Proof of Lemma \ref{lemma:balls}]
We note that the second estimate \eqref{est:E-balls-unit} is an immediate consequence of the first claim of the lemma and \eqref{eq:x_i}. 
We can hence focus on proving \eqref{est:E-balls-4pi} and for this first  consider indices $i\leq J-1$. As we can certainly assume that $\eps_2<\eps_1$ all of these indices correspond to radii 
 $\la_ir_i<\pi/2$ so we can view the restriction  of $v$ to $B_{2\la_ir_i}\setminus B_{\la_ir_i}(x_i)$ 
  as a map from  $[0,c_{\la_ir_i}]\times S^1$  as described above. From \eqref{eq:la_i} we know that there is 
$\hat s_i\in [0,c_{\la_ir_i}]$ with $\int_{S^1} \abs{\partial_\theta v(\hat s_i,\theta)}^2\leq 2(\log2)^{-1} \eps_2$. 
We can hence interpolate between $v(\hat s_i, \theta)$ and $\fint_{S^1} v(\hat s, \theta)$ on a unit cylinder in a way that requires energy no more than $C_0 \eps_2$ and that gives a map whose oscillation is no more than $C_0\sqrt{\eps_2}\leq \half$ where here and in the following $C_0$ stands for a universal constant that can change from line to line.  
Projecting this map onto $S^2$ hence gives a way to transition from $v(\hat s_i, \theta)$ to a constant using a map from a unit cylinder into $S^2$ of energy less than $C_0\eps_2$. 

To see that \eqref{est:E-balls-4pi} holds for $i=1$ we now let $\hat r_1\in [\la_1r_1,2\la_1r_1]$ be the radius that corresponds to $\hat s_1$ and use this map to transition from $v$ to a constant on the exterior of $ B_{\hat r_1}(x_1)$. This results in a map $V_1:S^2\to S^2$ which satisfies 
$$v=V_1 \text{ on }  B_{\hat r_1}(x_1) \text{ and } E(V_1,S^2\setminus B_{\hat r_1}(x_1))\leq C_0 \eps_2.$$
Similarly we can use the above map to also transition from $v$ to a constant  in the interior of $B_{\hat r_1}(x_1)$
to obtain a map $\tilde V_1:S^2\to S^2$ which satisfies 
$$\tilde V_1=v \text{ on } S^2\setminus B_{\hat r_1}(x_1) \text{ and } E(V_1,B_{\hat r_1}(x_1))\leq C_0 \eps_2.$$
As $C_0\eps_2$ is small, the parts of the maps that we glue in cannot change the degree so we must have that
 $\deg(V_1)+\deg(\tilde V_1)=\deg(v)$. Since $\de_v$ is small we can furthermore exclude the possibility that either of these degrees is negative and must thus have 
$$\de_{V_1}+\de_{\tilde V_1}\leq 
\de_v+C_0\eps_2\leq \bar\de+C_0 \eps_2
\leq \tfrac{\eps_1}{2k}$$
where the last estimate holds as we can choose $\eps_2=\eps_2(k):=\frac{\eps_1}{4kC_0}$ and restrict our attention to maps satisfying \eqref{ass:De_E_small} for some $\bar\de<\frac{\eps_1}{4k}$. 

Setting $k_1:=\deg(V_1)\geq 0$ we hence know that 
$$E(v, B_{\la_1 r_1}(x_1))\geq E(V_1)-E(V_1,S^2\setminus B_{\hat r_1}(x_1))\geq 4\pi k_1-C_0\eps_2\geq 4\pi k_1-\eps_1.
$$
This implies the claim for $i=1$ as the fact that $E(V_1)\geq
E(v,B_{r_1}(x))=\eps_1$ is strictly larger than $\de_{V_1}$ excludes the possibility that the degree 
$k_1$ of $V_1$ might be zero. 

We can then repeat this argument for $i=2,\ldots,J-1 $ with the map $\tilde V_{i-1}$ obtained in the previous step in the place of $v$: 
We get inductively that the energy defect of these maps is bounded by
$\de_{\tilde V_{i-1}}\leq \de_{v}+(i-1) C_0 \eps_2\leq \frac{i}{4k} \eps_1$, so will remain below $\frac{\eps_2}{2}$ while $i\leq k$. At the same time \eqref{eq:la_1} and  \eqref{eq:x_i}  ensure that for all such $i$ 
$$E(\tilde V_{i-1}, B_{\la_ir_i}(x_i))\geq 
E(v,  B_{\la_ir_i}(x_i)) \setminus \bigcup_{j\leq i-1} B_{2\la_j r_j}(x_j))\geq \eps_1-(i-1) \eps_2\geq \tfrac{3}{4}\eps_1.
$$ 
The resulting maps $V_i$ hence satisfy $E(V_i)>\de_{V_i}$ and must thus have positive degree. This ensures that the number of steps we have to carry out cannot exceed $k$ and that the above argument hence applies for all $i=1,\ldots, J-1$. It also ensures that $E(\tilde V_{J-1})\geq \frac34 \eps_1$ and hence that this final map $\tilde V_{J-1}$ must have positive degree which, combined with the smallness of the energy \eqref{est:rest-of-energy}
of $v$  on the complement of $\bigcup_{j\leq J} B_{\la_ir_i}(x_i)$, then immediately implies that the claim also holds for $i=J$.
\end{proof}

Having picked such a collection of balls $B_{r_i}(x_i)$ we now want to switch to a gauge in which we can extract the first harmonic map $\om_1$.

If the largest radius $r_J$ of the 
collection of balls $B_{r_i}(x_i)$, $i\in I_0:=\{1,\ldots,J\}$, obtained above is bounded away from zero by 
$r_J\geq\bar \La^{-1}\frac{\pi}{2}$ for 
\beq
\label{def:bar-La} 
\bar \La=\bar\La(k):= 4\La_0 (4K_0)^{4k}
\eeq
then we
work directly with $v_1=v$ and this collection of balls, so to keep the notation consistent with the second case set $s_i=r_i$ and $y_i=x_i$. 

Otherwise, we first want to switch viewpoint to ensure that at least one of the balls is of order $1$. Namely, if $r_J\bar \La< \frac{\pi}{2}$ then we use the  M\"obius transform
$M_1=M_{x_J,\bar \La r_J}:S^2\to S^2$ to scale $B_{\bar \La r_J }(x_J) $ up  to a hemisphere, which in turn scales $B_{r_J }(x_J) $
up to a ball $B_{s_J}(y_J)$ whose radius is bounded away from $0$ by a constant $c=c(k)>0$. 
Instead of working with $v$ we then consider the map 
$v_1=v\circ M_1^{-1}$  and the resulting collection 
 $B_{s_i}(y_i)=M_1(B_{r_i}(x_i))$  of balls. As the 
  claims of all of our results are invariant under M\"obius transforms, it suffices to establish them all for $v_1$. 

Rescaling the ball with maximal radius has the important advantage that we get energy estimates analogue to \eqref{est:item0-2-1} and \eqref{est:item0-2-2}. Namely, as the conformal factor $\rho_{M_1}$  is of comparable size at points whose distance is no more than $2\bar\La r_J$, compare Lemma \ref{lemma:dilation}, we have
\begin{lemma}
\label{lemma:new-balls}
There  exists a universal constant $K_0$ so that
we have bounds on the conformal factor of the above M\"obius transform $M=M_1$ of
\beq
\label{est:conf-factor}
K_0^{-1} \frac{s_i}{r_i} \leq \rho_M(x)\leq K_0 \frac{s_i}{r_i} 
\text{ on each } B_{2\bar\La r_i}(x_i), \quad i\in I_0\eeq
and can hence use that
\beq
\label{claim:ball-contained} 
B_{K_0^{-1} r\frac{s_i}{r_i}}(M(x))\subset M(B_{r}(x))\subset B_{K_0  r\frac{s_i}{r_i}}(M(x))\eeq
for each
$x\in B_{K_0^{-1}\bar \La r_i}(x_i)$, each $r\in (0,\bar \La r_i]$ and every $i\in I_0$.
Therefore the image $M(F_j^0)$, $j\in I_0$, of the set defined in \eqref{def:F0-j} is contained in 
\beq
\label{eq:F_j^1}
F_j^1:=\bigcup_{i\leq j} B_{K_0 \La_0 s_i}(y_i)
\eeq
 and 
the rescaled map $v_1$ satisfies energy estimates of the form 
 \beq\label{est:item1-2-2}
E(v_1,B_{K_0^{-1} s_j}(y)\setminus F_{j-1}^1)\leq \eps_1 \text{ for all } y \in B_{\bar\La K_0^{-2} s_j}(y_j) \text{ and } j\in I_0 \eeq
and 
\beq\label{est:item1-2-1}
E(v_1,S^2\setminus F_J^1)\leq \eps_1.
\eeq 
Additionally, 
 while the sequence of radii $s_i$ may no longer be ordered, we have 
\beq \label{est:item1-1}
s_i\leq K_0 s_j \text{ for all } i< j\text{ for which }B_{(2K_0)^{-2} \bar\La s_j}(y_j)\cap B_{(2K_0)^{-2} \bar\La s_i}(y_i)\neq \emptyset.\eeq
\end{lemma}

\begin{proof}
As we rescale a ball with radius $\bar \La r_J$ we know that the relation \eqref{est:rho-uniform} between the supremum and the infimum of $\rho_M$ over a ball can be applied to any ball with radius no more than $2\bar\La r_J$, so as $r_J$ is maximal, in particular on the balls $B_{2\bar\La r_i}(x_i)$, $i\in I_0$. 

As $M$ scales the diameter of $B_{r_i}(x_i)$ with a factor of $\frac{s_i}{r_i}$ we hence immediately obtain that
$$\tfrac{s_i}{r_i}\leq \sup_{B_{2\bar \La r_i}(x_i)} \rho_M \leq K_0 \inf_{B_{2\bar \La r_i}(x_i)}\rho_M\leq K_0\tfrac{s_i}{r_i},$$
i.e.~the first claim. As this yields upper bounds on both $\rho_M(x)$ and $\rho_{M^{-1}} (M(x))=\rho_M(x)^{-1}$ we can then immediately deduce \eqref{claim:ball-contained}.

Applied for $x=x_i$ and 
$r= K_0\La_0 r_i< \bar \La r_i$ this gives 
 $M(B_{\La_0r_i}(x_i))\subset B_{K_0\La_0s_i}(y_i)$ and hence ensures that indeed 
$M(F_j^0)\subset F_j^1$. 
Thus, \eqref{est:item1-2-1} is an immediate consequence of the corresponding estimate \eqref{est:item0-2-1}. 

Similarly, 
\eqref{est:item1-2-2}  follows from  \eqref{est:item0-2-2} and the fact that \eqref{claim:ball-contained} ensures that such balls $B_{K_0^{-1} s_j}(y)$
 correspond to balls $M^{-1}(B_{K_0^{-1} s_j}(y))$ in the previous viewpoint
  whose radii can be no more than $r_j$.

Finally, 
as \eqref{claim:ball-contained} ensures that 
$M^{-1}(B_{(2K_0)^{-2}\bar\La s_l}(y_l))\subset B_{(4K_0)^{-1} \bar \La r_l}(x_l)$ we must have that 
$\dist(x_i,x_j)\leq (4K_0)^{-1} \bar \La (r_i+r_j)$ for 
any indices $i$ and $j$ for which these balls are not disjoint. 
If $i<j$ and thus $r_i\leq r_j$ this 
ensures that $B_{r_i}(x_i)$ is contained in the set where \eqref{est:conf-factor} gives $\rho_M\leq K_0\frac{s_j}{r_j}$ so we must have that $s_i\leq  K_0\frac{s_j}{r_j}r_i\leq K_0 s_j$. 
\end{proof}

\begin{rmk}
\label{rmk:rescaling}
In later steps we will need to consider further rescalings to extract harmonic maps from clusters of highly concentrated balls. In these arguments we only ever rescale balls $B_{s_{\tilde J}} (y_{\tilde J})$ to order one which are in a given 
cluster, have not yet been captured and are chosen so that their radius $s_{\tilde J}$ is maximal among the radii of all balls with these two properties. 
 \\
This will ensure that estimates of the above form will be valid also after further such rescalings, though now only for the smaller collection of balls that have the above two properties, and 
 whose radius was hence bounded by $s_{\tilde J}$ before the rescaling. Namely for this collection of balls all of the above estimates and inclusions still apply, except that 
 we have to  increase the power with which the factors $K_0$, $K_0^{-1}$ respectively $(2K_0)^{-2}$ appear in \eqref{eq:F_j^1}, 
 \eqref{est:item1-2-2} and \eqref{est:item1-1} 
with each additional rescaling.  
As the number of rescalings we have to carry out is bounded by $k$, and as $\bar \La$ was chosen as in \eqref{def:bar-La}, this however does not affect our arguments. 
\end{rmk}

While the above comment only applies to balls in the given cluster that have not yet been captured, rather than to the full set $I_0$ of indices, this will be sufficient for the proofs of our main results. Indeed all other regions of the domain will be mapped into a very small neighbourhood of the antipodal point $-y_{\tilde J}$ in the gauge in which we rescale $B_{s_{\tilde J}}(y_{\tilde J})$ to order one, and will hence be cut out in our procedure anyway.

\subsection{Capturing the first rational map $\om_1$}
\label{subseq:first-step}
$ $\\
In a first step we now want to capture the behaviour of $v_1$ on all balls whose scale $s_i$ is not too small. To this end we set $\al_1:=(4k)^k \al$, $\al>0$ the given exponent in our main results, and 
split 
 these balls into 
 a collection of balls $B_{s_i}(y_i)$, $i\in I_1$, that we want to capture in the current step, and a (potentially empty) collection $B_{s_i}(y_i)$, $i\in I_*:=I_0\setminus I_1$,  of highly concentrated balls in a way that ensures that
\beq
\label{def:I1}
\min_{I_1} s_i\geq c (\de_v)^{3k\al_1}\text{ for some } c=c(k)>0 \text{ while } \min_{I_1} s_i\geq (\de_v)^{-3\al_1} \max_{I_1^c}  s_i.
\eeq
We can e.g.~select such a $I_1$  by first adding the index $J$ of the maximal radius $s_{J}\geq c=c(k)>0$
 to $I_1$
 and then continuing to add indices by decreasing order of $s_i$ until we find an index for which the quotient between the radius we just picked and the next smaller radius is greater than $(\de_v)^{-3\al_1}$.

Having chosen $I_1$ in this way we then define
\beq
\label{def:r1-star}
s_1^*:= \de_v^{2\al_1} \min_{I_1} s_i \text{ and } S_1^*:= \de_v^{\al_1} \min_{I_1} s_i 
\eeq
and note that these radii are so that 
\beq
\label{est:radii-order}
\max_{I_*} s_i \ll s_1^*\ll S_1^*\ll \min_{I_1} s_i
\eeq
where each $a\ll b$ is to be understood as $a\leq C \de_v^{\al_1}$ for a constant $C=C(k,\al)$. 

We note that while $s_1^*$ will be very small, it will still satisfy 
\eqref{ass:mui} so we will be able to work with a metric $g_1$ which rescales the balls $B_{s_1^*}(y_i)$, $i\in I_*$ to unit size and also rescales the collection of balls $B_{s_i}(y_i)$, $i\in I_1$, in this way, compare \eqref{def:g_1} below.

This will result in a flow that is well controlled away from the highly concentrated set $F_{*,\frac14}$ which we define by
\beq 
\label{def:HC}
F_{*,\la}:= \bigcup_{j\in I_*} B_{\la s_1^*}(y_j) \text{ for factors } \la >0 .\eeq
Before we turn to the analysis of this flow, we first want to derive two more energy estimates that involve these sets $F_{*,\la}$ in the place of unions of balls $B_{s_i}(y_i)$, $i\in I_*$. To this end we note that
we must have that $j<i $
for all indices 
\beq\label{rel:indices}
i\in I_1 \text{ and } j\in I_* \text{ for which } B_{s_1^*}(y_j)\cap M(B_{\la_ir_i(x_i))})\neq \emptyset
\eeq
as \eqref{est:item1-1} would otherwise imply that  $s_i\leq K_0 s_j$ which is impossible since $s_j\ll s_1^*\ll s_i$. For such indices $i\in I_1$ and $j\in I_*$ we furthermore know from \eqref{est:conf-factor} that $\rho_{M^{-1}}\leq K_0 \frac{r_i}{s_i}$ on $ B_{s_1^*}(y_j)$ which allows us to deduce that
$M_1^{-1} (B_{2s_1^*}(y_j))$ is contained in a ball around $x_j$ whose radius is bounded by  $2 K_0s_1^* \frac{r_i}{s_i}$, and hence  much smaller than the radius $r_i$ of the ball $B_{r_i}(y_j)$ that we cut out of $B_{\la_ir_i}(x_i)$ in \eqref{est:E-balls-unit}. 

This lower bound \eqref{est:E-balls-unit} on the energy
from Lemma \ref{lemma:balls} hence ensures that 
 \beqa
 \label{est:Ev1-lower}
 E\big(v_1, M_1\big[B_{\la_i r_i}(x_i)\setminus\bigcup _{j\in I_1, j<i}B_{\la_i r_i}(x_i)\big] \setminus F_{*,2}\big)\geq 4\pi k_i-i\eps_1 \text{ for each }  i\in I_1,
\eeqa
 and thus that the energy of $v_1$ outside of the highly concentrated set $F_{*,2}$ is at least 
\beq
\label{eq:energy-main-1}
E(v_1, S^2\setminus F_{*,2})\geq 4\pi\sum_{i\in I_1}k_i -k^2\eps_1\geq  4\pi\sum_{i\in I_1}k_i-\tfrac{\pi}{4}.
\eeq
%
As $s_i\ll s_1^*$ for $i\in I_*$ we know that $M_1(\bigcup_{I_*}B_{\la r_i}(x_i))\subset F_{*,\frac14}$ and can hence also use 
 \eqref{est:E-balls-4pi} to obtain a 
 lower bound on the energy on the highly concentrated set of 
$$E(v_1,F_{*,\frac14})\geq E(v, \bigcup_{I_*}B_{\la r_i}(x_i))\geq 
 4\pi\sum_{I_*}k_i-k\eps_1$$
and hence an upper bound on the energy of $v_1$ away from $F_{*,\frac14}$ of 
\beq
\label{eq:energy-main-2}
E(v_1, S^2\setminus F_{*,\frac 14})\leq 4\pi\sum_{I_1}k_i +k\eps_1+\de_v\leq 4\pi\sum_{ I_1}k_i+\tfrac{\pi}{4}.
\eeq
We now consider the
weighted harmonic map flow for the domain metric 
\beq
\label{def:g_1}
g_1=\rho^2 \gz:= 
(1+\sum_{i\in I_1} \rho_{y_i,r_i}^2+\sum_{j\in I_{*}} \rho_{y_j,s_1^*}^2) g_{S^2}\eeq
which will be well controlled away from  $F_{*,\frac14}$ since we 
 were able to rescale subsets of $S^2\setminus F_{*,\frac14}$ with the right weights.
To be more precise, we show

\begin{lemma}\label{lemma:first-flow}
Let  $v_1$ and $g_1$ be as above,  let $F_{*,\la}$ be as defined in \eqref{def:HC} and set 
$$F_{I_1,\la}:= \bigcup_{i\in I_1} B_{\la K_0\La_0 s_i}(y_i) \text{ for } \la>0.$$ 
Then the solution $u$  of the weighted harmonic map flow
$$\partial_t u=\tau_{g_1}(u), \quad u(0)=v_1
$$ is so that 
\beq
\label{est:step1-complement}
E(u(t),S^2\setminus (F_{I_1,2}\cup F_{*,\frac14}))\leq \eps_1+C \xi_v \text{ for all } t\in [0,\infty]
\eeq
and so that for every $y\in F_{I_1,4} 
$ there exists $j\in I_1$ with
\beq
\label{est:step1-small-balls} y\in  B_{8K_0\La_0 s_j}(y_j) \text{ and } 
E(u(t),B_{(2K_0)^{-1} s_j}(y)\setminus F_{*,\frac14}) \leq \eps_1+C \xi_v \text{ for all } t\in [0,\infty].
\eeq
Hence all potential singularities of the flow, be it at finite time or as $t\to \infty$, must occur on the closure of the highly concentrated set $F_{*,\frac14}$ and we furthermore have
\beq
\label{est:step1-lower}
\abs{E(u(t), S^2\setminus F_{*,1})-4\pi\sum_{i\in I_1} k_i }\leq \tfrac\pi 4+C\xi_v \text{ for all } t\in [0,\infty].
\eeq
\end{lemma}

Here and in the following $C$ is allowed to depend on the fixed numbers $\al$ and $k$ unless specified otherwise and we continue to use the convention that 
all results and arguments are to be understood to hold provided $v$ satisfies \eqref{ass:De_E_small} for a sufficiently small number $\bar \de=\bar \de(k,\al)>0$. 

\begin{proof}
All of these claims are obtained by considering cut-off energies $E_\phi$ for functions $\phi$ that are constructed out of cut-off functions 
 $\phi_{x,r}\in C_{c}^\infty(B_{r}(x))$ with $\phi_{x,r}\equiv 1$ on $B_{r/2}(x)$ for which Lemma \ref{lemma:cut offs} provides uniform bounds on $\norm{d\phi}_{L^\infty(S^2,g_1)}$.
 
 Namely, to prove \eqref{est:step1-complement} we set
$$\phi:=\prod_{i\in I_1} (1-\phi_{y_i,2K_0\La_0 s_i})\prod_{i\in I_*}  
(1-\phi_{y_i,\frac14 s_1^*})
$$ 
and note that 
since $K_0\La_0\leq\bar\La(k)$ we can apply Lemma \ref{lemma:cut offs} to see that 
\beq
\label{est:norm-phi-1}
\norm{d\phi}_{L^\infty(S^2,g_1)} \leq C=C(k).
\eeq 
This cut-off function 
has support in 
$S^2\setminus (F_{I_1,1}\cup F_{*,\frac18})$ which we know to be  contained in $ S^2\setminus F_J^1$, since
 $K_0\La_0s_i\ll s_1^*$ if $i\in I_*$. 
The estimate \eqref{est:item1-2-1} obtained in Lemma \ref{lemma:new-balls} hence ensures that 
$E_\phi(0)\leq \eps_1$  and inserting this as well as 
\eqref{est:norm-phi-1} into the upper bound
\eqref{est:energy-cut-upper} on the evolution of the cut-off energy immediately implies that 
$$E_\phi(t)\leq \eps_1+C\xi_v \text{ for all } t\in [0,\infty].$$ 
As $\phi\equiv 1$ on $S^2\setminus (F_{I_1,2}\cup F_{*,\frac14})$ 
this yields the first claim \eqref{est:step1-complement} of the lemma.

Given 
$y\in F_{I_{1},4}$ we let $i\in I_1$ be so that $y\in B_{4K_0\La_0s_i}(y_i)$ and then choose $j\in I_1$ to be the minimal index in $I_1$ for which 
 $B_{s_i}(y)\cap  B_{4K_0\La_0 s_j}(y_j)\neq \emptyset$. 
As $j$ can be no larger than $i$ we obtain from \eqref{est:item1-1}  that $(K_0)^{-1} s_j\leq s_i$ and hence know that 
$B_{(K_0)^{-1} s_j}(y)\subset 
B_{s_i}(y)$ must be disjoint from any ball $ B_{4K_0\La_0 s_{a}}(y_{a})$ for indices $a\in I_1$ with $a<j$. Therefore
\beq
\label{rel:raining}
B_{K_0^{-1} s_j}(y)\setminus F_{j-1}^1= B_{K_0^{-1} s_j}(y)\setminus\bigcup_{l\leq j-1, l\in I_*} B_{K_0\La_0 s_{l}}(y_{l})\supset B_{K_0^{-1} s_j}(y)\setminus F_{*,\frac18} \eeq
where the second relation holds as $ s_l\ll s_1^*$ for  $l\in I_*$. 

Hence \eqref{est:item1-2-2} gives $E(v_1, B_{K_0^{-1}s_j}(y)\setminus F_{*,\frac18})\leq \eps_1$ and applying the above argument for a $\phi$ that is supported on this set and so that
$ \phi\equiv 1$ on the smaller set  $B_{(2K_0)^{-1} s_j}(y)\setminus F_{*,\frac14}$  yields the second claim \eqref{est:step1-small-balls} of the lemma.

As these energies are all bounded by  $\eps_1+C\de_v \leq 2\eps_1<\pi/4$
 and 
 as the formation of a bubble requires energy at least $4\pi$ we can hence exclude the possibility that any singularities, be it at finite or infinite time, form outside of the closure of the highly concentrated set $F_{*,\frac14}$.  

This now allows us to use the two sided estimate \eqref{est:energy-cut} for cut-off energies
that are supported outside of $F_{*,\frac14}$ 
 and hence to obtain the final claim \eqref{est:step1-lower} from the corresponding upper and lower bounds \eqref{eq:energy-main-1} and \eqref{eq:energy-main-2} on the energy of $v_1$.  
\end{proof}

This completes Step 1 of the construction of our first harmonic map $\om_1$. At this point we have a limiting harmonic map $\tilde \om_1=u(\infty)$ which satisfies the above energy estimates and is so that  
\beq
\label{123}
\de_{\tilde\om_1}\leq \de_{v} \text{ and } \norm{\tilde \om_1-v_1}_{L^2(S^2,g_1)}\leq C \xi_v\eeq
since it was obtained by a weighted flow for which \eqref{ass:mui} holds. 
To carry out Step 2, we now show that 

\begin{lemma}
\label{lemma:energy-density}
There exist numbers $\La$ and $C$ that only depend on $k$ and $\al$ so that the energy density of the obtained harmonic map $\tilde \om_1:S^2\to S^2$ is bounded by 
\beq\label{claim:step1-energy-density}
\abs{\na \tilde \om_1}^2\leq C(\sum_{I_1} \rho_{y_i,s_i}^2+\sum_{I_*}\rho_{y_i,s_1^*}^2) \text{ on  }  S^2\setminus F_{*,\La}
 \eeq
\end{lemma}

\begin{proof}
We want to apply 
Lemma \ref{lemma:rational} so have to describe how the corresponding collections of balls are chosen. For the first collection we simply use  $J_1=I_1$, points $z_i=y_i$ and radii $R_i:= 2K_0\La_0s_i$.

 To obtain the second collection of balls we 
 note that since 
 we are dealing with no more than $k$ balls we can always fix $\La=\La(k)$ so that there will be a subset $J_2$ of $I_*$ and radii $R_i\in[s_1^*,\frac14\La s_1^*]$  so that
the balls $B_{2R_j}(y_j)$, $j\in J_2$, are pairwise disjoint and so that 
$F_{*,\frac14} \subset \bigcup_{I_2} B_{R_j}(y_j)$.

The energy estimates  \eqref{est:step1-complement} respectively \eqref{est:step1-small-balls} obtained in the above 
Lemma \ref{lemma:first-flow} then immediately imply that the assumptions
 \eqref{ass:lemma-rat-1} and \eqref{ass:lemma-rat-2} are satisfied for $d:=(2K_0)^{-1}\La_0^{-1}$. 
 
 It hence remains to check that $E(\tilde\om_1, B_{2R_l}\setminus B_{R_l }(y_l))\leq 2\eps_1$ for each $l\in J_2$ and we note that 
these
annuli are disjoint from $F_{*,\frac14}$.
Hence 
this follows from \eqref{est:step1-complement} if $ B_{2R_l}(y_l)\cap F_{I_1,2}=\emptyset$ while for all other indices $l\in I_2$ we must have 
 $y_l\in F_{I_1,4}$ so can obtain this claim from \eqref{est:step1-small-balls} and the fact that $R_l\leq Cs_1^*\ll s_i$ for all $i\in I_1$.

We can thus apply Lemma \ref{lemma:rational} to deduce that
\beq
\abs{\na \tilde \om_1}^2\leq C \sum_{I_1\cup I_2} \rho_{z_i,R_i}^2
\text{ on } S^2\setminus \bigcup_{j\in I_2} B_{4R_j}(y_j).
\eeq
This yields the claimed bound \eqref{claim:step1-energy-density}  as 
$S^2\setminus F_{*,\La}$ is contained in the above set and as the 
conformal factors in the above estimate are bounded by $C\rho_{y_i,s_i}$ respectively $C\rho_{y_i,s_1^*}$ for some $C=C(k)$ since the centres of the dilations agree while the radii are of comparable size, compare also Remark \ref{rmk:rho-convention}.
\end{proof}

We are now finally in the position to complete the argument that allows us to extract our first rational map.

The simplest case is when $I_{*}=\emptyset$, i.e.~when we have no highly concentrated balls that we were unable to scale up with the correct factor. 
In this case Step 2 and Step 3 are not needed since  Lemma \ref{lemma:first-flow} ensures that the flow remains smooth for all times and converges smoothly on all of $S^2$ to a harmonic limit  $\tilde \om_1$ 
whose the energy density is bounded by 
$$\abs{\na\tilde \om_1}^2\leq C \sum\rho_{y_i,s_i}^2 \text{  on all of } S^2$$
thanks to Lemma \ref{lemma:energy-density}. In this case we can hence simply set $\om_1:=\tilde \om_1$ and $\Om_1=S^2$ and obtain a single harmonic map $\om_1$ of degree $k$ which satisfies the key properties \ref{key:1}-\ref{key:4}.

So suppose instead that the set of highly concentrated balls is not empty. 
In this case the flow might form singularities on $F_{*,\half}\supset \overline{F_{*,\frac14}}$ and we might hence lose some of the energy that was initially concentrated on $F_{*,\half}$. However, we cannot expect that all of the energy of $F_{*,\half}$ is dissipated this way and hence 
cannot expect that the limit of the first flow already has the required properties. 

We hence need to carry out Step 2 and cut out 
all highly concentrated 
parts of $\tilde \om_1$. 
To this end we partition $I_*$ into disjoint subsets $I_*^1,\ldots,I_*^m$ each of which corresponds to a cluster of highly concentrated balls whose distance is of order $O(S_1^*)$. 
Namely, we can choose such a partition of $I_*$ so that there
 are radii $R_*^j\in [S_1^*,C S_1^*]$,  $C=C(k)$,  
 and points $y_*^j$ so that the balls $B_*^j:=B_{R_*^j}(y_*^j)$ are pairwise disjoint and so that 
$\bigcup_{I_{*}^j }
B_{S_1^*}(y_i)\subset B_{R_*^j}(y_*^j) \text{ for each } i\in I_*^j$ and 
$S_1^*=\de_v^{-\al_1}s_i^*$ as in \eqref{def:r1-star}.

We then define the first subset $\Om_1$ of our partition of $S^2$ by
\beq
\label{def:Om_1}
\Om_1:= S^2\setminus \bigcup_{j\in I_2} B_*^j.
\eeq
We note that this set is separated from the highly concentrated balls  $B_{s_1^*}(y_i)$, $i\in I_*^j$, by annuli 
of the form $B_{R_*^j}\setminus B_{(1-d)R_*^j}(y_*^j)$ for some $d=d(k)>0$ in the sense that 
\beq
\label{rel:nice-annuli}
\bigcup_{i\in I_*^j} B_{s_1^*}(y_i) \subset B_{(1-d) R_*^j}(y_*^j) \text{ while }  B_{R_*^j}(y_*^j) \subset S^2\setminus \Om_1.
\eeq
This follows since for each $i\in I_*$ and $j\in I_2$
\beq
\label{est:dist-circles}
\dist(\partial B_*^j,y_i)\geq S_1^*= (\de_v)^{-\al_1} s_1^* \eeq
is far larger than $s_1^*$ and of the same order as $R_*^j$. 
We will use \eqref{rel:nice-annuli} later in Section \ref{subsec:next-step} to argue that we will never capture the same region of the domain twice, but for now complete the extraction of the first rational map.  

For this we use that \eqref{est:dist-circles} ensures that 
the highly concentrated conformal factors $\rho_{y_i,s_1^*}$
we used in the definition of $g_1$  are small compared to $(R_*^j)^{-1}$ on the circles $\partial B_*^j$ in the sense that for all $i\in I_*$
\beq 
\rho_{y_i,s_1^*} \leq Cs_1^* \dist(\partial B_*^j,y_i)^{-2}\leq 
C (\de_v)^{\al_1}  (R_*^j)^{-1}  \text{ on } \partial B_*^j, \quad j\in I_2
\eeq
compare \eqref{est:rho-approx}. The analogue statement is trivially true for the other conformal factors $\rho_{y_i,s_i}$, $i\in I_1$, used in the definition of $g_1$  as these are bounded uniformly by 
\beq
\label{est:rho_i-all}
\norm{\rho_{y_i,s_i}}_{L^\infty(S^2)}\leq Cs_i^{-1}\leq C(\de_v)^{\al_1} (S_1^*)^{-1}\leq C (\de_v)^{\al_1} (R_*^j)^{-1}.
\eeq
As \eqref{est:dist-circles} also ensures that the circles $\partial B_*^j$ are in the region $S^2\setminus F_{*,\La}$ where the energy density of $\tilde \om_1$ is controlled by \eqref{claim:step1-energy-density} we hence obtain that 
\beq
\label{est:rat-circles}
\abs{\na \tilde \om_1}\leq C (R_*^j)^{-1} (\de_v)^{\al_1} \text{ on each } \partial B_*^j=\partial B_{R_*^j}(y_*^j),
\eeq
which gives bounds on the oscillation and the tangential derivative of $\tilde \om_1$ along   $\partial B_*^j$ of
\beq
\label{est:osc-Bj}
\osc_{\partial B_*^j} \tilde\om_1 \leq \int_{\partial B_*^j}\abs{\partial_\tau \tilde \om_1} dS_{\gz}\leq C \de_v^{\al}.
\eeq
We can hence modify
 $\tilde \om_1$ in the following way: 
Let 
 $h^j:B_*^j\to \R^3$ be the  harmonic function that
agrees with $\tilde \om_1$ on $\partial B_*^j$  and note that 
 \eqref{est:osc-Bj} and the maximum principle ensure that 
$h^j$ takes values in a small neighbourhood of $S^2$. 
We can hence replace $\tilde \om_1$ by $\frac{h^j}{\abs{h^j}}$ on each  such ball to obtain a new function $\tilde v_1\in H^1(S^2,S^2)$ which agrees with $\om_1$ on $\Om_1$ 
and whose energy 
on $S^2\setminus \Om_1$ is very small. Namely, for each  $j\in I_2$ we can use that 
 the tangential and normal derivative of a harmonic function from a disc are related by $\abs{\partial_\tau h^j}=\abs{\partial_n h^j}$ 
 to obtain from \eqref{est:osc-Bj} that 
\beqa
\label{est:energy-balls-cut-out}
E(\tilde v_1, B_*^j
)&\leq C \int_{B_*^j} \abs{\na h^j}^2\leq 
C\osc_{\partial  B_*^j}h^j \int_{\partial  B_*^j}\abs{\partial_n h^j }dS_\gz\\
&\leq C\osc_{\partial  B_*^j}\tilde \om_1 \int_{\partial  B_*^j}\abs{\partial_\tau \tilde \om_1 }dS_\gz
\leq C (\de_v)^{2\alpha_1}.\eeqa
On the other hand, we now claim that the amount of energy that we cut out is 
\beq \label{est:energy-F} 
\sum_{I_2}E(\tilde \om_1, B_*^j)\geq 4\pi(\deg(\tilde \om_1)-\deg(\tilde v_1)) -C (\de_v)^{2\alpha_1}.
\eeq
To see that this holds we let 
$\tilde h^j$ be the harmonic function 
from $S^2\setminus B_*^j$ to $\R^3$ which agrees with 
$\tilde \om_1$ on  $\partial B_{1}^j$ and use that $E(\tilde h^j)=E(h^j)\leq  C (\de_v)^{2\alpha_1}$ thanks to the conformal invariance of the energy and \eqref{est:energy-balls-cut-out}. Extending  
 $\tilde \om_1$ from $B_*^j$ to $S^2$ using $\frac{\tilde h^j}{\abs{\tilde h^j}}$ 
hence yields  functions $w_j:S^2\to S^2$ with 
$$0\leq 4\pi\deg(w_j)\leq E(w_j)\leq E( \tilde \om_1,  B_*^j)+C(\de_v)^{2\alpha_1}.$$ These maps must furthermore be must be so that 
$\deg(\tilde\om_1)=\deg(\tilde v_1)+\sum_j\deg(w_j)$
 since we only ever glue in maps which have small energy  and oscillation and which can hence not affect the degree. Combined this yields   
\eqref{est:energy-F}. 

As the energy defect is non-increasing along the flow we can use this estimate \eqref{est:energy-F} to bound the energy defect of our new map by 
\beqa
\label{est:energy-defect-new}
\de_{\tilde v_1}&=\sum_{j\in I_2}
 E(\tilde v_1,   B_*^j)
 +E(\tilde\om_1)-\sum_{j\in I_2}
E(\tilde \om_1, B_*^j)-4\pi \deg(\tilde v_1)\leq \de_{\tilde \om_1} +C(\de_v)^{2\alpha_1}   \\
&\leq \de_v+C (\de_v)^{2\alpha_1} \leq 2 \de_v.
\eeqa
As $S_1^*\ll \min_{I_1} s_i$ we furthermore know from 
\eqref{est:step1-complement} and \eqref{est:step1-small-balls} that 
$E(\tilde \om_1,B_*^j\setminus F_{*,\frac14})\leq \eps_1+C\xi_v$ is small and hence that most of the energy that we cut out was indeed concentrated on $F_{*,\frac14}\subset F_{*,1}$.
Since the energy of $\tilde \om_1$ on the complement of $F_{*,1}$ is close to $4\pi \sum_{i\in I_1} k_i$ as described in  \eqref{est:step1-lower} we hence know that energy of our new map $\tilde v_1$ must be so that 
$\abs{E(\tilde v_1)-4\pi \sum_{i\in I_1} k_i}\leq \frac\pi4+2k \eps_1+C\xi_v\leq \pi
$
which, combined with the smallness of $\de_{\tilde v_1}$, ensures that 
$$\deg(\tilde v_1)= \sum_{i\in I_1} k_i.$$ 
We furthermore note that while the $L^2(S^2,g_1)$ distance of $\tilde v_1$ from $\tilde \om_1$ is large, this new map $\tilde v_1$ is close to $\tilde \om_1$, and hence close to the original map $v_1$, with respect to the weaker $L^2(S^2,\tilde g_1)$ norm
that we obtain if we use the weaker metric
 \beq
\label{def:g_1-tilde}
\tilde g_1=\rho_1^2 \gz:=(1+\sum_{i\in I_1} \rho_{y_i,r_i}^2) g_{S^2}\eeq
that only contains weights of balls $B_{s_i}(y_i)$ that are in $\Om_1$. 

 Namely as $\tilde \om_1\equiv\tilde v_1$ on $\Om_1$ and  
as \eqref{est:rho_i-all} ensures that the weighted Area of $S^2\setminus \Om_1$ with respect to this new metric is controlled by 
 $\Area_{\tilde g_1}(\Om_1^c)\leq C(\de_v)^{2\al_1}$
 we get from \eqref{123} that
 \beq
 \label{est:new-w-dist}
  \norm{\tilde v_1-v_1}_{L^2(S^2,\tilde g_1)}\leq  \norm{ v_1-\tilde \om_1}_{L^2(S^2,g_1)}+C (\de_v)^{\al_1}\leq  C \xi_v. 
\eeq
We can now finally carry out Step 3, i.e.~evolve this new map $\tilde v_1$ with 
the weighted harmonic map flow with respect to this new metric $\tilde g_1$. All of the above arguments still apply, and indeed  simplify significantly, as this new initial map has now small energy on all of $S^2\setminus F_{I_1,2}$. We are hence dealing with a situation as in the first simpler case considered above where the initial map has no highly concentrated regions at all. 

As explained above, we hence know that this second flow
 remains smooth for all times, that it converges smoothly to a limiting harmonic map $ \om_1$ with
 $$\deg(\om_1)=\deg(\tilde v_1)=\sum_{I_1}k_i,$$
 and that the
$L^2(S^2,\tilde g_1)$ distance of $\om_1$  to $\tilde v_1$, and hence to $v_1$, is controlled by $C\xi_v$. Since we have no highly concentrated regions we can now also use Lemma \ref{lemma:energy-density} to obtain a bound on the 
energy density of
\beq
\label{est:lim1-0}
\abs{\na \om_1}^2\leq C\rho_1^2 \text{ for } \rho_1^2:= \sum_{i\in I_1} \rho_{y_i,s_i}^2 \text{ on all of } S^2.
\eeq
Altogether this establishes that the rational map $\om_1$ and the domain $\Om_1$ defined in \eqref{def:Om_1} are so that all key properties \ref{key:1}-\ref{key:4} required for the proof of Theorem \ref{thm:gauge} hold true.

\subsection{Extracting all other rational maps}
\label{subsec:next-step}$ $
\\
To complete the proof of Theorem \ref{thm:gauge} it hence remains to explain how the above argument can be applied also to extract all other harmonic maps $\om_\beta$ and domains $\Om_\beta$ from the clusters of highly concentrated balls 
if $I_*\neq \emptyset$.
 
We can analyse each cluster $\bigcup_{i\in I_*^a} B_{s_i}(y_i)\subset B_*^a$, $a=1,\ldots,m_1$, separately, so can focus on how to extract the harmonic maps that correspond to indices of $I_*^1$.

To extract the next harmonic map we let $J_2\in  I_*^1$
be so that $s_{J_2}=\max_{I_*^1} s_i$ and carry out the above argument again, except that we rescale further with the M\"obius transform $M_2$ that dilates $B_{\bar \La s_{J_2}}(y_{J_2})$ to a hemisphere, i.e.~consider the map $v_2=v_1\circ M_2^{-1}$ and the balls $B_{\tilde s_i}(\tilde y_i)=M_2(B_{s_i}(y_i))$, and now work with exponent $\al_2=(4k)^{k-1} \al=\frac{\al_1}{4k}$. 

The main difference to the previous argument is that we can now focus on the balls $B_{\tilde s_i}(\tilde y_i)$, $i\in I_*^1$, of the cluster we are analysing, as all other clusters and the previously obtained $\Om_1$ are scaled down to highly concentrated balls that are contained in a very small neighbourhood of the antipodal point 
$-y_{J_2}$ of the centre of our new dilation.

To be more precise,  as our construction ensures that 
$s_{J_2}\leq (\de_v)^{\al_1} s_1^*$ and that each point $y_j$, $j\in I_*^1$, has distance at least $S_1^*= (\de_v)^{-\al_1} s_1^*$ from both $\Om_1$ and from all other clusters $B_*^{j\neq 1}$
we obtain from \eqref{rel:dist-Mob} that 
\beq
\label{est:image-M2}
M_2(\Om_1\bigcup_{j\neq 1} B_*^j)\subset B_{C (\de_v)^{2\al_1}}(-y_{J_2}) \text{ for some }C=C(k).
\eeq
$M_2$ hence maps all of these sets into a ball around $-y_{J_2}$ whose radius is small compared to $(\de_v)^{3k \al_2}$. This ensures that  in the new gauge 
all indices that are not in $I_*^1$ correspond to highly concentrated 
balls which end up in the cluster $B_{**}^{1,0}=B_{R_{**}}(-y_{J_2})$ that forms at the antipodal point $-y_{J_2}$ whose radius $R_{**}$ will scale like $S_2^*\gg (\de_v)^{3k \al_2}$.  These balls will hence be cut out when we carry out the 3-step-procedure to extract the next harmonic map which ensures that the corresponding domains are disjoint. 

On the other hand, for the indices $i\in I_*^1$
we get the analogue 
 of Lemma \ref{lemma:new-balls} as explained in Remark \ref{rmk:rescaling} since we again 
rescale using the largest ball of the relevant collection.

We can hence proceed exactly as above,
now splitting $I_*^1$ into a subset of indices $I_{1}^1$ that we capture in the current argument, and the complementary set of indices $I_{**}^1$ that correspond to highly concentrated balls of radius less than $\de_v^{\al_2}s_2^*$ for $s_{2}^*:= \de_v^{2\al_2}\min_{I_1^1} \tilde s_{i}\gg \de_v^{3k \al_2}$. As above we then split these 
highly concentrated balls $B_{\tilde s_i}(y_i)$, $i\in I_{**}^1$, further into  clusters that are contained in disjoint balls whose radius is of order $S_2^*=\de_v^{-\al_2} s_2^*
$
where we now distinguish between 
 \begin{itemize}
 \item The (potentially empty) set of balls $B_{\tilde s_i}(y_i)$, $i\in I_{**}^{1,0}\subset I_{**}^1$, that end up in a ball $B_{**}^{1,0}=B_{R_{**}}(-y_{J_2})$ around the 
 antipodal point whose radius is of order $S_2^*$. As observed above, this ball also contains the images of $\Om_1$ and of $B_*^j$, $j\neq 1$.  
 \item A (potentially empty) collection of other clusters 
corresponding to index sets $I_{**}^{1,b}$, $b\geq 1$, that will be contained in balls
 $B_{**}^{1,b}$ whose radii are of order $S_2^*$ and which are disjoint from each other and from $B_{**}^{1,0}$. 
\end{itemize} 
Setting $\tilde \Om_2:= S^2\setminus \bigcup _{b\geq 0} B_{**}^{1,b}
$ we can then repeat the above argument and change $v_2=v_1\circ M_2^{-2}$ into a harmonic map $h$ which is close to $v_2$ on $\tilde \Om_2$. The pull-back $\om_2=h\circ M_2$ to the first gauge hence describes the behaviour of $v_1$ on the set $\Om_2:= M_2^{-1}(\Om_2)$ as claimed in \ref{key:1}-\ref{key:4} and we note that the construction ensures that $\Om_2$ is disjoint from $\Om_1$ as $M_2(\Om_1)$ is contained in the ball $B_{**}^{1,0}$ that we cut out in the above argument.

We then want to iterate this procedure to also extract the harmonic maps that correspond to the indices $I_{**}^{1,b}\subset I_*^1$, $b\geq 0$, of balls that we have not yet captured and that we cut out in this second gauge. 

For $b\geq 1$, i.e.~for clusters  that form away from the antipodal domain, we can argue exactly in the same way as above (now with exponent $\al_3=\frac{\al_2}{4k}$ and rescaling $v_2$ further to a map $v_2\circ M^{-1}$) and the resulting harmonic maps can be seen as bubbles that form on top of $\om_2$.

To extract the balls that correspond to indices in $I_{**}^{1,0}$, i.e.~balls  that 
in the first gauge  were in the cluster  from which we extracted $\om_2$, but were mapped very close to the antipodal point with the second rescaling, we need to proceed differently.  
 These regions do not correspond to bubbles that form on top of $\om_2$ but rather to bubbles that form alongside $\om_2$ and at a large distance (compared to the radii of the corresponding balls) in the first layer of bubbles that form on top of the base map $\om_1$. 
 
  We hence extract
  these balls by going back to the first viewpoint, i.e.~consider an alternative rescaling of the map $v_1$, rather than a further rescaling of $v_2$. 
  
In the first gauge these balls are given by subsets $B_{s_i}(y_i)$, $i\in I_{**}^{1,0}$, of $B_*^1$ so had radius $s_i\leq s_{J_2}$. 
  We now again select the index $i=J_3$ in $I_{**}^{1,0}$ 
so that the corresponding radius $s_{J_3}$ is maximal among the radii of all relevant balls, i.e.~now so that $s_{J_3}\geq s_i$ for all $i\in I_{**}^{1,0}$. We then rescale from the first gauge using the M\"obius transform $M_3$ that rescales $B_{\bar\La s_{J_3}}(y_{J_3})$ to a hemisphere, i.e.~consider the new map $v_3=v_1\circ M_3^{-1}$, and set again  $\al_3=\frac{\al_2}{4k}$.  
  
The argument that we used to extract $\om_2$ now works in exactly the same way, except that we need to carry out one additional step in which we argue that also 
 the domain $\Om_2=M_2^{-1}(S^2\setminus \bigcup B_{**}^{1,b})$ that  we extracted in the second step now gets 
 mapped into a very small 
neighbourhood of the antipodal point $-y_{J_3}$ of this new rescaling $M_3$.

This step is needed as $\Om_2$ is in the same cluster as 
$B_{s_{J_3}}(y_{J_3})$ so \eqref{est:image-M2} is not sufficient to ensure that we do not capture this region again in the current step.

Instead we use that the highly concentrated balls are always separated from the domain that we extract by an annulus as described in \eqref{rel:nice-annuli}, and that this ensures that the distance between $\Om_2$ and $B_{s_{J_3}}(y_{J_3})$ is very large compared to $s_{J_3}$.  
To be more precise, \eqref{rel:nice-annuli}
 ensures that the set 
 $M_2(B_{s_{J_3}}(y_{J_3}))$ that represents this ball in the second gauge 
was separated from the domain $\tilde \Om_2=M_2(\Om_2)=S^2\setminus \bigcup B_{**}^{1,b} $ 
that we extracted in that step by the annulus
$$\tilde A=B_{R_{**}}\setminus B_{(1-d)R_{**}}(-y_{J_2}).$$ 
We also recall that the radius $R_{**}$ of $B_{R_{**}}(-y_{J_2})=B_{**}^{1,0}$ is of order $ S_2^*\leq C \de_v^{\al_2}=C\de_v^{4k \al_3}$ so far smaller than $\de_v^{3k\al_3}$. 

As the centre of this annulus is antipodal to the centre of the dilation $M_2$ we used to go from the first to the second gauge, it gets pulled back to an annulus of the form 
 $A=M_2^{-1}(\tilde A)=B_{R_1}\setminus B_{R_2}(y_{J_2})$ when we switch back to the first view point. Here the radii $R_1$ and $R_2$ must be so that 
 $$R_1-R_2\geq c \frac{s_{J_2}}{(R_{**})^2} d R_{**} \geq c \frac{s_{J_2}}{R_{**}}\geq c \de_v^{-\al_2} s_{J_2},$$
as the conformal factor of $M_2^{-1}$ scales like $\frac{s_{J_2}}{\dist(y,-y_{J_2})^2}\sim \frac{s_{J_2}}{(R_{**})^2} $ on $\tilde A$, compare \eqref{est:rho-approx}. Here $c=c(k)>0$ denotes a constant that is allowed to change in every step. 

In the first gauge, this annulus $A$ separates $\Om_2$ from all balls $B_{s_i}(y_i)$ with $i\in I_{**}^{1,0}$, and hence in particular from the ball $B_{s_{J_3}}(y_{J_3})$ that we use for the next rescaling. We thus get a lower bound on the distance of this balls from $\Om_2$ of 
$$\dist(B_{s_{J_3}}(y_{J_3}),\Om_2)\geq R_2-R_1 \geq c \de_{v}^{-\al_2} s_{J_2}\geq c \de_{v}^{-\al_2}s_{J_3}\gg \de_{v}^{-3k\al_3}s_{J_3}.$$
Thus
$M_3(\Om_2)$ is contained in a ball around the antipodal point $-y_{J_3}$ of radius less than $\de_{v}^{-3k\al_3}$, 
so will be cut out when we carry out our three-step-process for $v_3=v_1\circ M_3^{-1}$. We can thus  
proceed exactly as above to extract $\om_3$.

We can iterate the above argument until we have captured all indices in $I_0$ and hence obtain a collection of maps $\om_\beta$ with total degree $k$ which satisfy the key properties \ref{key:1}-\ref{key:4} on the corresponding domains $\Om_\beta$ for  exponents $\al_{\beta}$ which are all so that $\al_\beta\geq (4k)^{-k}\al_1=\al$.

This completes the proof of Theorem \ref{thm:gauge}.

\subsection{Proof of Theorem \ref{thm:3} based on Theorem \ref{thm:gauge}}
\label{subsec:last-step} $ $

To show that  Theorem \ref{thm:gauge} implies Theorem \ref{thm:3} we
 first recall that \eqref{claim:thm-gauge-1} and \eqref{claim:Omi-gauge}  give 
$$\int_{S^2\setminus \Om_i}\abs{\na \om_i}^2 \dvgz\leq C\int_{S^2\setminus \Om_i}\rho_i^2 \dvgz=C\Area_{\rho_i^2 \gz}(S^2\setminus \Om_i)\leq C\de_v^{2\al},$$
i.e.~yield the second claim \eqref{claim:thm3-2} of Theorem \ref{thm:3}. 

It hence remains to discuss how the quantitive $H^1$ estimate \eqref{claim:thm3-1} claimed in our first main result can be obtained from this estimate and the $L^2$ rigidity estimate obtained in 
Theorem \ref{thm:gauge}, or to be more precise, from the estimate \eqref{claim:thm1} that is  stated in Corollary \ref{cor} and that is an immediate consequence of Theorem \ref{thm:gauge}.

This argument is similar to an argument used by Topping in his analysis \cite{Topping-deg-1} of degree $1$ maps between spheres. As we will see that this proof can be modified in a way that it applies also for maps from any closed surface $\Si$ into any closed smooth target $N\hookrightarrow \R^K$, we
carry out the proof in this more general setting. 

So suppose that $v: \Si\to N$ is so that there exists a collection of harmonic maps $\om_i:\Si\to N$ 
with $E(v)=\sum_i E(\om_i)+\de_v$ for which the analogues of 
\eqref{claim:thm1} and \eqref{claim:thm3-2} hold, i.e.~for which 
$$\int_{\Si} \abs{\na \om_i}^2\abs{v-\om_{i}}^2 dv_g\leq C \xi_v^2\,\text{ and } \int_{\Si\setminus \Om_i} \abs{\na\om_i}^2 dv_g\leq C\de_v^{2\al}$$
for each $i$, for a partition $\{\Om_i\}$ of $\Si$ and some $\al\geq 1$. 
Then we have 
\beqas
\thalf \sum_i\int_{\Om_i}\abs{\na(v-\om_i)}^2&
=\thalf \int_{\Si} \abs{\na v}^2-\thalf \sum_i\int_{\Om_i} \abs{\na\om_i}^2+\sum_i\int_{\Om_i} \na \om_i\na(\om_i-v)\\
&= E(v)-\sum_i E(\om_i) +\sum_i\int_{\Si} \na \om_i\na(\om_i-v)+\text{err}_1
\eeqas
for an error term that is bounded by 
$$\abs{\text{err}_1}\leq C\sum_i \norm{\na\om_i}_{L^2(\Si\setminus \Om_i)}(E(\om_i)+E(v))\leq C \de_v^\al.$$
As $\om_i$ is a harmonic map into $N$ it satisfies 
$-\Delta \om_i=A(\om_i)(\na \om_i,\na \om_i)$, $A$ the second fundamental form of $N\hookrightarrow \R^K$. As $A$ maps into the normal bundle $T^\perp N$, we thus have
\beqas
\abs{\int_{\Si} \na \om_i\na(\om_i-v)}&=
\abs{\int_{\Si} A(\om_i)(\na  \om_i,\na  \om_i)(\om_i-v)}\leq C\int_{\Si} \abs{\na\om_i}^2\abs{P_{\om_i}^\perp (\om_i-v)} \\
&\leq C \int_{\Si} \abs{\na\om_i}^2\abs{\om_i-v}^2\leq C\xi_v^2,
\eeqas
where the penultimate step follows since
the projection of the difference between two points $p_{1,2}\in N$ onto the normal space is bounded by 
$\abs{P^{\perp}_{p_1}(p_1-p_2)}\leq C \abs{p_1-p_2}^2$. 

Combined this gives the claimed estimate of 
$$\sum_i\int_{\Om_i}\abs{\na(v-\om_i)}^2 dv_g\leq \de_v+C\xi_v^2+C\de_v^{\al}\leq C\de_v\abs{\log\de_v}.$$

\section{Proof of Theorem \ref{thm:2}}
\label{sect:sharp}
To prove Theorem \ref{thm:2}  we consider a family of maps $v_{a,\mu}$, $a\in(0,\bar a]$, $\mu\geq \bar\mu$, of degree $k$ that are obtained by 
scaling a harmonic map $\si_a$ with a large factor $\mu$ and gluing it onto another harmonic map $\si_0$ at a point $p\in S^2$. 

To get the optimal result we want to choose $\si_0$ and $\si_a$ so that their combined degree is $k$, so that the distance between  $\si_0(p)$ and 
$\si_a(0,0,-1)$, i.e.~the asymptotic value of $\si_a$ in stereographic coordinates, is given by the parameter $a$ and so that these maps are not branched in the points where we glue, i.e.~so that both 
 $\na \si_0(p)$ and $\na \si_a(0,0,-1)$ are non-zero.

To fix ideas we can hence e.g.~work with the harmonic maps which are given in (complex) stereographic coordinates on the domain as 
$$\si_a(z)=R_a \pi(\tfrac{1}{z}) \text{ and } \si_0(z)=\pi(z(1+z^{k-2})),
$$ 
$R_a$ the rotation 
by the angle $a\in (0,\bar a]$ around the $y_2$-axis in $\R^3$ and glue the maps at $z=0$. Here and in the following $\pi:\C\to S^2\hookrightarrow \R^3$ denotes the inverse stereographic projection which is given by 
\beq
\label{def:pi}
\pi(x_1+ix_2):= \big(\tfrac{2x_1}{1+\abs{x}^2}, \tfrac{2x_2}{1+\abs{x}^2}, \tfrac{1-\abs{x}^2}{1+\abs{x}^2}\big).
\eeq
Using the gluing construction developed by the author in Section 2 of \cite{R-2bubbles} we obtain maps $v_{a,\mu}$
 which are so that 
\beq
\label{eq:vdelta-ball-1}
v_{a,\mu}(z)=R_a\pi(z(1+z^{k-2})+\tfrac{1}{\mu z})+O(a\abs{z}) \text{ for } \abs{z}\leq r_1:=\mu^{-1+d}
\eeq
and 
\beq
\label{eq:vdelta-ball-2}
v_{a,\mu}(z)=\pi(z(1+z^{k-2})+\tfrac{1}{\mu z})+O(\tfrac{a}{\mu \abs{z}}) \text{ for } \abs{z}\geq r_0:=\mu^{-d},
\eeq
and 
for which the analogue expansions
for the derivatives also hold, now with errors of order $O(a)$ respectively $O(\frac{a}{\mu\abs{z}^2})$. 
Here 
$d\in (0,\frac 14)$ is a small, but fixed exponent and corresponds to a choice of $f_\mu=\mu^{-d}$ in the construction in \cite{R-2bubbles} and we also note that the oscillation and energy of these maps on the annulus  $A=\DD_{r_0}\setminus \DD_{r_1}$ are small, namely of order $\osc_A v=O(a+\mu^{-d})$ and $E(v,A)=O(\frac{a^2}{\log\mu})$. 
We note that as $\mu\to \infty$ these maps converge to a bubble tree with base map $\si_0$ and bubble $\si_a$. 

While there would be other, simpler ways of gluing maps to get a sequence that converges to such a bubble tree, a key aspect of the construction in  \cite{R-2bubbles} is that it yields maps whose 
 energy defect is small, namely given by 
 \beq
\label{claim:sharp-delta}
\de_{v_{\mu,a}} \leq C\frac{a^2}{\log\mu} \text{ for some } C=C(k),
\eeq
 see  
estimate (4.15) of \cite{R-2bubbles}.
For comparison, if we were to simply interpolate between $\si_a(\mu\cdot)$ and $\si_0$ on a suitable annulus, or first cuts off these maps to constants and then interpolate between these constants, this would result in an energy defect of order $\mu^{-1}+\frac{a^2}{\log\mu}$. 
While this would already be sufficient to get an example of maps satisfying \eqref{claim:intro}, compare also Remark \ref{rmk:no-rigidity} below, such a construction would not allow us to establish that our rigidity estimates are sharp. Indeed to prove Theorem \ref{thm:2} we will later need to choose sequences for which $a^{-1}$ scales like $\mu$ and for which the energy defect $\de_v\sim \mu^{-2}(\log\mu)^{-1}$ is hence far smaller than the rate of $O(\mu^{-1})$ that one would obtain from simpler gluing constructions.

\begin{rmk}
\label{rmk:sharp}
We note that the leading order term in the energy defect $\de_v$ of the maps constructed in \cite{R-2bubbles} corresponds to the 
amount of energy that is needed to transition from $\si_a(\infty)$ to $\si_0(0)$ on the annulus between the region where the bubble is concentrated and the bulk of the domain. It is this simple feature of being able to transition between different constants $c_1$ and $c_2$ on an annulus with a map whose energy scales like $ \frac{\abs{c_1-c_2}^2}{\abs{\log(r_1/r_0)}}$ which is reflected in the fact that the sharp rate of the quantitative rigidity estimates for maps of degree at least $2$ is not $C\de_v$ but rather $C\de_v\abs{\log\de_v}$.

We recall that a rate of the form $\de_v\abs{\log\de_v}$ appears also in a conjecture that was stated in \cite{Deng} and that was based on the construction of a sequence of maps which violate \eqref{est:quant-deg-1}. We however point out that the conjectured estimate in \cite{Deng} does not hold as rigidity fails if one compares general maps of degree $k\geq 2$  with rational maps of this degree, compare Remark \ref{rmk:no-rigidity} below and \eqref{claim:intro} in the introduction. 
 \end{rmk}

We will see that the distance of these maps $v_{a,\mu}$ from any rational map $\om$ with the same degree is at least of order $a=\dist(\si_a(\infty),\si_0(0))$ and note that this reflects the fact that the zeroth order terms in the expansions \eqref{eq:vdelta-ball-1} and \eqref{eq:vdelta-ball-2} have distance of order $a$ while for holomorphic functions the meanvalue over circles is constant.

\begin{lemma}
\label{lemma:sharp1}
There exist numbers $\bar\mu<\infty$, $\bar a>0$ and $c_0>0$ so that for any $a\in (0,\bar a]$, any $\mu\geq \bar \mu$ 
and any rational map $\om:S^2\to S^2$ with degree $k$
we have 
\beq
\label{claim:dist-sharp-1}
\dist_\om(\om,v_{a,\mu})^2:= \int_{S^2}\abs{\na \om}^2 \abs{
\om-v_{a,\mu}
}^2 \dvgz \geq c_0\cdot a^2
\eeq
as well as 
\beq
\label{claim:dist-sharp-2}
\int_{S^2}\abs{\na (\om-v_{a,\mu})}^2 \dvgz \geq c_0\cdot a^2.
\eeq
\end{lemma}

\begin{rmk}
\label{rmk:no-rigidity}
For fixed $a\neq 0$ and $\mu_n\to \infty$ we can hence obtain an example of a sequence of maps $v_n=v_{a,\mu_n}$ whose energy defect tends to zero but whose distance to the set of rational maps with the given degree is bounded away from zero. 
\end{rmk}

On the other hand, the distance of $v_{a,\mu}$ from any collections $\{\om_i\}$ of multiple rational maps which have total degree $k$, and hence each have $\deg(\om_i)<k$, is bounded by 
\begin{lemma}
\label{lemma:sharp2}
There exist $\bar\mu<\infty$, $\bar a>0$ and $c_0>0$ so that for any $a\in (0,\bar a]$ and any $\mu\geq \bar \mu$
\beq
\label{claim:dist-sharp-3}
\dist_\om(\om,v)^2=\int_{S^2}\abs{\na \om}^2 \abs{v_{a,\mu}-\om}^2 \dvgz \geq c_0 \mu^{-2} \log\mu
\eeq
for any  rational map $\om:S^2\to S^2$ with degree $1\leq \deg(\om)\leq k-1$.
Furthermore, 
\beq \label{claim:dist-sharp-4}
\sum_i\int_{\Om_i} \abs{\na (v_{a,\mu}-\om_i)}^2  \geq c_0 \mu^{-2}
\eeq
for any collection $\om_1,\ldots, \om_m$ of $m\geq 2$ rational maps with $\deg(\om_i)\geq 1$ and $\sum_i \deg(\om_i)=k$ and any partition $\{\Om_i\}$ of $S^2$. 
\end{lemma}

Combined, these two lemmas allow us to prove that 
Theorem \ref{thm:2} holds for sequences for which the deviation $a$ of the two maps in the target is of the same order as the radius $r\sim \mu^{-1}$ of the ball in the domain where the bubble concentrates. 

To be more precise, we can e.g.~consider $v_n=v_{\mu_n,a_n}$ for parameters $a_n\in (0,\bar a]$ and  $\mu_n\geq \bar \mu$ with $\mu_n=a_n^{-1}\to \infty$ and note that 
 \eqref{claim:dist-sharp-1} and \eqref{claim:dist-sharp-3} ensure that $d_{\om_n}(v_n,\om_n)\geq \sqrt{c_0}\mu_n^{-1}$ for all $n$ and for all rational maps $\om_n$ with $1\leq \deg(\om_n)\leq k$. 
As \eqref{claim:thm2} would be trivially true if we had 
 $\log\mu_n< \frac14\abs{\log\de_{v_n}}$ and hence $d_\om(v_n,\om_n)\geq c\de_{v_n}^{\frac14}$, we can then use the bound \eqref{claim:sharp-delta} on the energy defect to see that
$$\dist_{\om_n}(v_n,\om_n)^2\geq  \frac{c_0}{ \mu_n^2 \log\mu_n} \log\mu_n\geq  c  \de_v \abs{\log\de_v} \text{ for some } c=c(k)>0
$$
 as claimed in Theorem \ref{thm:2}. 

As \eqref{claim:dist-sharp-2} and \eqref{claim:dist-sharp-4} provide the analogue bounds on the 
$\dot H^1$ distance of the maps $v_n$ from any collection of rational maps with total degree $k$ we obtain the second claim of the theorem by the same argument.

This hence reduces the proof of Theorem \ref{thm:2} to the proofs of 
these two lemmas.

\begin{proof}[Proof of Lemma \ref{lemma:sharp1}]
We first explain why both claims of the lemma follow once we have shown that the maps $v=v_{a,\mu}$
satisfy
\beq \label{est:may0}
\int \abs{\na v}^2 \dist_{\C P^1}^2(\om,v) \geq c a^2 \text{ for some } c=c(k)>0
\eeq
for all degree $k$ rational maps $\om$ for which 
\beq
\label{wlog:sharp}
\int\abs{\na (\om-v)}^2 \leq a^2.
\eeq
Here and in the following we let 
$\dist_{\C P^1}(p_1,p_2)=\min(\dist_{S^2}(p_1,p_2),\dist_{S^2}(p_1,-p_2))$ be the distance that identifies antipodal points and compute integrals over $S^2$ and with respect to the standard metric unless specified otherwise. We also use the convention that all arguments are to be understood to hold for $a\in (0,\bar a]$ and $\mu\geq \bar \mu$ for suitably small $\bar a$ and large $\bar\mu$.

Indeed, since $\deg(\om)=\deg(v)$ and hence $E(\om)=E(v)-\de_v$, we can first apply the arguments from Section \ref{subsec:last-step} and then use \eqref{claim:sharp-delta} to see that 
\beq\label{est:may1}
\int\abs{\na (\om-v)}^2\leq \de_v+C \int \abs{\na\om}^2 \abs{\om-v}^2 \leq C\tfrac{a^2}{\log\mu}+C \int \abs{\na\om}^2 \abs{\om-v}^2 
\eeq
and hence deduce that the two claims of the lemma are trivially true if $\int\abs{\na (\om-v)}^2> a^2$. 

Furthermore, combining \eqref{est:may1} with 
\beqas
\int \abs{\na v}^2 \dist_{\C P^1}^2(\om,v) 
&\leq \int\abs{\na \om}^2  \dist_{\C P^1}^2(\om,v) +\int\abs{\na(v-\om)} (\abs{\na v}+\abs{\na \om})   \dist^2 _{\C P^1}(\om,v)\\
&\leq 
\tfrac32 \int \abs{\na\om}^2  \dist_{\C P^1}^2(\om,v)  +\thalf\int \abs{\na v}^2   \dist_{\C P^1}^2(\om,v) +C \int\abs{\na (w-v)}^2 
\eeqas
shows that proving \eqref{est:may0} is equivalent to establishing a lower bound of 
\beq \label{est:may2}
\int \abs{\na \om}^2 \dist_{\C P^1}^2(\om,v) \geq c a^2
\eeq
which then of course immediately implies the first claim of the lemma. 

At the same time \eqref{est:may2} also implies the second claim of the lemma since the fact that $\om$ is (weakly) conformal ensures that we have a pointwise bound of 
\beq
\label{est:may-hallo}
\abs{\na \om -\na v}\geq c \cdot  \dist_{\C P^1}(\om,v) \abs{\na \om}\eeq
for a universal $c>0$.
This basic fact follows since the distance between any matrix $A\in M_{3\times 2}(\R)$ whose columns are given by an on-basis of the tangent space $T_pS^2\subset \R^3$ at a point $p$ and any other matrix $B\in M_{3\times 2}(\R)$ whose columns are contained in $T_{\tilde p}S^2$ is bounded from below in terms of the distance of the corresponding tangent planes and hence by  $
\abs{A-B}\geq c\, \dist_{\C P^1}(p,\tilde p)$.

It thus remains to prove \eqref{est:may0} for maps which satisfy 
 \eqref{wlog:sharp}. For such maps we can use that the difference between the energy of  $\om$  and the energy of $v$ on any ball is no more than $Ca^2$, and hence, for suitably chosen $\bar a$, less than the number 
 $\eps_1$ obtained in Lemma \ref{lemma:rational}. This lemma hence implies that 
\beq
\label{est:sharp-proof-3}
\abs{\na \om}^2 \leq C(\abs{\na \pi}^2+\abs{\na \pi_\mu}^2)
\eeq
for some $C=C(k)$, for $\pi_\mu=\pi(\mu\cdot)$ and for all maps $\om$ that we have to consider.

We therefore know that the oscillation of $\om$ over an annulus of the form $A=\DD_{\La^{-1}}\setminus \DD_{\La \mu^{-1}}$ is small if $\La$ is large. 
The same statement holds true for $v$ (if also $\bar a$ is small and $\bar \mu$ is large)
and we note that $v(A)$ is contained in a small neighbourhood of $\pi(0)$. We also remark that 
the maps $v$ and $\om$ must be close on the bulk of the domain $S^2$ since
\eqref{wlog:sharp} implies that also  $\int \abs{\na \om}^2 \dist_{\C P^1}(v,\om)^2$ needs to be small and since the
smallness of the $\dot H^1$ distance of $v$ to $\om$ excludes the possibility that 
$v$ is close to $-\om$ in $L^2(S^2)$. 

While these statements could all be carefully quantified, all we need from this discussion is that 
we can fix $\La$ so that $\om(A)$ and $v(A)$ are contained in 
$ B_{\frac{\pi}{4}}(\pi(0))$ for all maps $v=v_{a,\mu}$ (with sufficiently small  $a$ and large  $\mu$) and 
all $\om$ which satisfy \eqref{wlog:sharp}.

On $A$ we hence have $\dist_{\C P^1}(v,\om)=\dist_{S^2}(v,\om)$ and know that the functions $\tilde v= \pi^{-1} \circ v$ and $\tilde \om:= \pi^{-1} \circ \om:\hat \C\to \hat \C$ which represent $v$ and $\om$ in stereographic coordinates on the target are so that $\abs{\tilde\om}, \abs{\tilde v}\leq \tan(\frac\pi8)\leq C$. This allows us to conclude that 
$$\abs{\tilde v-\tilde \om}\leq C \abs{v-\om}\leq C \dist_{\C P^1}(v,\om) \text{ and } \abs{\na \tilde v}\leq C\abs{\na v} \text{ on } A,$$
so 
\eqref{est:may0} follows if we establish that 
\beq\label{est:vde-deg2-2}
\int_{A} \abs{\na \tilde v}^2 \abs{\tilde v-\tilde \om}^2 dv_\C\geq ca^2.\eeq
To this end we can use that the 
 meromorphic
function $\tilde \om:\hat\C\to \hat \C$ is indeed holomorphic in $A$ since $\abs{\tilde \om}\leq \tan(\frac\pi8)\leq C$ is uniformly bounded on this set. Hence the
meanvalue
$\fint_{S^1} \tilde\om(r e^{i\theta}) d\th $ is independent of $r\in [\La\mu^{-1}, \La^{-1}]$. 
We now want to argue that this means that a lower bound of the form 
\beq
\label{claim:may}
\abs{\fint_{S^1} (\tilde\om-\tilde v )(r e^{i\theta}) d\th }\geq \tfrac{a}{4} \text{ for } r\in I_i \eeq
must be valid at least on one of the intervals $I_0=[\thalf\La^{-1},\La^{-1}]$ or $I_1=[\La \mu^{-1},2\La \mu^{-1}]$.

To see this we write for short $p(z)=z(1+z^{k-2})+\frac{1}{\mu z}$ and note that  
 \eqref{eq:vdelta-ball-2} yields
$$\fint_{S^1} \tilde v(re^{i\th}) d\th= \fint_{S^1} p(r e^{i\theta}) d\theta +O(a\mu^{-1})=O(a\mu^{-1}) \text{ for } r\in I_0.$$
On the other hand 
 \eqref{eq:vdelta-ball-1} 
ensures that $\abs{\tilde v-\pi^{-1}(R_a\pi(p(z))}=O(a\mu^{-1})$ on $\abs{z}\leq 2\La\mu^{-1}$. As $z\mapsto \pi^{-1}(R_a\pi(p(z))$ is holomorphic on $A$  and as 
$\abs{\pi(p(z))-\pi(0)}\leq C\mu^\mhalf$ for $\abs{z}=\mu^\mhalf$ 
we hence see that
\beqas
\abs{\fint_{S^1} \tilde v(re^{i\th}) d\th }&\geq \abs{\fint_{S^1} \pi^{-1}(R_a(\pi(p(re^{i\th}))) d\th }-Ca\mu^{-1}\\
&= \abs{\fint_{S^1} 
\pi^{-1}(R_a(\pi(p(\mu^{\mhalf} e^{i\th}))d\theta }-Ca\mu^{-1}
\geq \abs{\pi^{-1}(R_a(\pi(0))}-Ca\mu^{\mhalf}
\eeqas 
 for $r\in I_1$. Combined this ensures that 
$$\abs{\fint_{S^1} \tilde v(re^{i\theta})- \tilde v(\tilde re^{i\theta})d\theta}\geq  \abs{\pi^{-1}(\sin a, 0, \cos a)}-Ca\mu^{\mhalf}=\tfrac{\sin a}{1+\cos a}- Ca\mu^{\mhalf}\geq  
\tfrac a2 $$ for all $r\in I_1$ and $\tilde r\in I_0$. As $\fint \tilde \om (r e^{i\theta})$ is constant in $r$ we deduce that 
\eqref{claim:may} holds at least on one of these intervals $I_0$ or $I_1$. 
As $\abs{\na v}\geq c$ for $\abs{z}\in I_0$ and $\abs{\na v}\geq c\mu$
for $\abs{z}\in I_1$ this 
 yields the required bound \eqref{est:vde-deg2-2} and hence completes the proof of the lemma. 
\end{proof} 

\begin{proof}[Proof of Lemma \ref{lemma:sharp2}]
We will prove the two claims separately and note that for some of these arguments it will be more convenient to view the maps $v$ and $\om$ as maps from the domain $S^2$ rather than from $\C$ and that we will do so without switching notation, though will indicate in which viewpoint we work.

To 
 establish the first claim it suffices to 
  consider sequences of maps $v_n=v_{a_n,\mu_n}$ for which $\mu_n\to \infty$  and $a_n\to a_{\infty}\in [0,\bar a]$ and show that the claim holds for a subsequence of any given sequence of rational maps $\om_n$ with degree in $\{1,\ldots, k-1\}$.   

Given such sequences $v_n$ and $\om_n$ of maps from $S^2$ to $S^2$  we can use standard compactness results to pass to a subsequence so that $\om_n$ converges smoothly to a limiting harmonic map $\om_\infty$ away from a finite number of points in $S^2$ where bubbles form. 
If any of these bubbles form at a point $p\in S^2$ which is different from the point $\pi(0)\in S^2$ where $v_n$ forms a bubble, then $\dist_{\om_n}(\om_n,v_{n})$ remains bounded away from zero. Similarly, if the limiting harmonic map $\om_\infty$ is non-constant, but does not agree with $\si_0$, then  $\dist_{\om_n}(\om_n,v_{n})$ remains bounded away from zero so the claim is trivially true. 

It hence remains to consider the two cases where $\om_\infty=\si_0$ respectively where $\om_\infty$ is constant and where all of the energy concentrates near $\pi(0)\in S^2$. 

In the former case the maps $\om_n$ must indeed converge smoothly to $\si_0$ on all of $S^2$, as $\deg(\om_n)\leq k-1=\deg(\om_\infty)$  excludes the formation of any bubble. 
As $\si_0$ is not branched at $\pi(0)$ we hence obtain  a pointwise lower bound of $\abs{\na \om_n}\geq c>0$ that is valid on a fixed size ball around $\pi(0)$ in $S^2$ and for all sufficiently large $n$.

Switching back to stereographic coordinates on the domain we hence know that $\abs{\na \om_n}\geq \tilde c>0$ on a fixed disc around $z=0$ and thus in particular on the discs
 $\DD_{r_1(\mu_n)}$ where $v_n$ is described by \eqref{eq:vdelta-ball-1}. 
The smooth convergence of $\om_n$ of course also ensures that the oscillation of $\om_n$ over any circle $\partial \DD_r \subset A_n:=\DD_{\mu_n^{-1+d}}\setminus\DD_{\mu_n^{-1}}
$  is of order $O(r)=O(\mu_n^{-1-d})$ and we recall from 
\eqref{eq:vdelta-ball-1} that also $\abs{v_n-R_a\pi(\frac{1}{\mu_n z})}=O(\abs{z})=O(\mu_n^{-1-d})$ on this annulus.  
 
We then note that for any $q=(q_1,q_2,q_3) \in \R^3$ and any $r\in (0,1)$ the set of points in $\partial \DD_r$ at which the second estimate in
$$\abs{\pi(x)-q}\geq \abs{\tfrac{2x}{1+\abs{x}^2}-(q_1,q_2)}\geq \abs{\tfrac{2x}{1+\abs{x}^2}}\geq \abs{x}$$ is valid consists of at least half of the circle $\partial\DD_r$. Combined this ensures that 
\beqas
\dist_{\om_n}(\om_n,v_n)^2&\geq c\int_{A_n} \abs{v_n-\om_n}^2\geq  
c\inf_{q\in\R^3} \int_{A_n}\abs{\pi(\frac{1}{\mu_n z})-q}^2 dv_\C-C\mu_n^{-2(1-d)}\Area(A_n) \\
&\geq c \int_{A_n} \abs{\mu_n\abs{z}}^{-2} dv_\C -C 
 \mu_n^{-4(1-d)} \geq c \mu_n^{-2}\log\mu_n
\eeqas
which establishes the claim \eqref{claim:dist-sharp-3} in the case where $\om_n$ converges smoothly to $\si_0$. 

This finally leaves us to prove  \eqref{claim:dist-sharp-3} in the case where all of the energy of $\om_n:\C\to S^2$ concentrates near $z=0$.
In this case the compactness theory of harmonic maps ensures that
after passing to a subsequence there are points $z_n^i\to 0$,   bubble scales $\la_n^i\to \infty$ and non-constant harmonic maps $\si_i:\C\to S^2$, $i=1,\ldots,\bar m\geq 1$, so that 
$$\om_i(z)-\sum (\si_i(\la_n^i(z-z_n^i))-\si_i(\infty))\to \om_\infty\equiv \text{const} $$
in both $H^1$ and $L^\infty$ on a fixed size disc $\DD_{c}$ around the point $z=0$. 

As above we can easily see that $\dist_{\om_n}(\om_n,v_n)$ does not tend to zero if there is more than one bubble, or if the points $z_i^n$ are so that $\mu_n\abs{z_i^n}\to\infty$ or if the scales 
$\la_i^n$ are not of the same order as $\mu_n$. 
After switching viewpoint via $z\mapsto \frac{1}{\mu_n z}$ we are hence back in the situation where we are dealing with a sequence of rational maps that converges smoothly to a harmonic map which is not branched at $z=0$ and can hence repeat the above argument to deduce that 
$d_{\om_n}(v_n,\om_n)^2\geq c \mu_n^{-2}\log\mu_n$. 

This completes the proof of the first claim of the lemma.

The key observation in the proof of the second claim \eqref{claim:dist-sharp-4}  is that this estimate is trivially true if 
\beq
\label{fast-fertig}
\thalf\sum_i \int_{S^2\setminus \Om_i} \abs{\na \om_i}^2 \dvgz\geq c \mu^{-1} \text{ for some } c=c(k)>0.\eeq
 Indeed, as $\sum E(\om_i)=4\pi k\leq  
E(v)=\sum_iE(v,\Om_i) $ this ensures that 
\beqas
c\mu^{-1}&\leq \thalf\sum \int_{S^2\setminus \Om_i} \abs{\na \om_i}^2 \dvgz= \sum E(\om_i)-\thalf \sum \int_{\Om_i}\abs{\na\om_i}^2\\
&\leq \thalf \sum \int_{\Om_i}
\abs{\na v}^2-\abs{\na\om_i}^2 =
\sum \int_{\Om_i}(\na v-\na \om_i)(\na v+\na \om_i)\\
&\leq C \sum \norm{\na v-\na \om_i}_{L^2(\Om_i,\gz)}.
\eeqas
Importantly \eqref{fast-fertig} holds if the collection of rational maps is given by two maps $\om_1$ and $\om_2$, which are so that one of them approximates the base and one approximates the bubble as in this case we will have 
\beq
\label{est:fast-fertig} 
\int_{A_*\setminus \Om_i} \abs{\na \om_i}^2\geq c \mu^{-1} , \quad 
 A_*:=\{z: \thalf \mu^{-1/2}\leq \abs{z}\leq 2\mu^{-1/2}\}\eeq
 for at least one of $i=1,2$.
 
Indeed, 
standard compactness results for harmonic maps imply that if $\norm{\na (\om_1-\si_0)}_{L^2(S^2)}$ is small then $\om_1$ is indeed $C^1$ close to $\si_0$
and hence so that $\abs{\na \om_1}\geq c=c(k)>0$ on $A^*$. The 
analogue estimate of $\abs{\na \om_2}\geq c$  for maps for which $ \norm{\na (\om_2-\na \si_a(\mu\cdot))}_{L^2(S^2)}$ is small follows by symmetry. 
As  $\Om_1$ and $\Om_2$ are disjoint, we know that 
$\Area(A^*\setminus \Om_i)\geq \half \Area(A^*)\geq c \mu^{-1}$ for at least one of $i=1,2$, and hence that \eqref{est:fast-fertig} holds for at least one of $i=1,2$.

This is sufficient to obtain the second claim of the lemma as 
 a simple compactness argument shows that the left hand side of \eqref{claim:dist-sharp-4} is bounded away from zero uniformly
for any other collection of rational maps and any partition for which \eqref{fast-fertig} does not hold. 
\end{proof}

\section{Proof of Proposition \ref{prop:1}} \label{sec:Prop}
We finally provide a proof of the \Loj estimate claimed in Proposition \ref{prop:1} that bounds the energy defect in terms of the weighted tension. We recall that the \Loj estimate \eqref{est:Loj-standard} that involves the tension with respect to the standard metric was proven by Topping in \cite{Topping-rigidity} and will use several key ideas from that proof.

In this section we view maps $u:S^2\to S^2$ as maps from $\C$ to $S^2$ unless indicated otherwise 
and will exploit the well known fact that for maps into the sphere the energy has a natural splitting into its holomorphic and antiholomorphic part. Namely, we can use that the energy density (with respect to the standard metric on the domain $\C$) can be expressed as 
$$\thalf   \big[\abs{\partial_{x_1} u}^2+\abs{\partial_{x_2} u}^2\big]=\abs{u_z}^2 +\abs{u_{\bar z}}^2$$
if we view the target $S^2$ as a complex manifold and consider the corresponding holomorphic and antiholomorphic derivatives $u_z=\half(u_{x_1}-i u_{x_2})$ and $u_{\bar z}=\half (u_{x_1}+i u_{x_2})$. Equivalently, we can compute $\abs{u_z}^2$ and $\abs{u_{\bar z}}^2$ 
 by introducing stereographic  coordinates also on the target to view $u$ as a map $\tilde u:= \pi^{-1}\circ u:\C \to \C$ and set
$\abs{u_z(z)}^2= \half \abs{d\pi(\tilde u(z))}^2 \abs{\tilde u_z(z)}^2$, $\pi:\C\to S^2$ the inverse stereographic projection defined in \eqref{def:pi}.

Combined with the conformal invariance of the energy with respect to the domain metric we hence get the well known splitting of
$ 
E(u)=E_{\partial}(u)+E_{\bar \partial}(u) $ 
for
$$E_{\partial}(u)= \int_{\C}\abs{u_z}^2 dv_{\C} \text{ and } E_{\bar\partial}(u)=\int_{\C}\abs{u_{\bar z}}^2 dv_{\C}.
$$
As in \cite{Topping-rigidity} will also exploit that the degree of maps $u:\C\to S^2$ is given by 
$$4\pi \deg(u)=E_{\partial}(u)-E_{\bar \partial }(u)$$
and hence that the energy defect of maps with positive respectively negative degree is given by
$\de_u=2E_{\bar \partial} (u)$ respectively $\de_u=2E_{\partial} (u)$. 

As precomposing with a reflection changes the sign of the degree it hence suffices to prove
\beq
\label{claim:Loj-hol-energy-1}
 E_{\partial}(u)\leq C(1+\max \abs{\log(r_i)}) \norm{\tau_g(u)}_{L^2(S^2,g)}^2 \text{ for some } C=C(d)
 \eeq
for 
 metrics $g$ as in \eqref{def:g} and 
 maps $u:\C\to S^2$ with  $-d\leq \deg(u)\leq 0$ 
for which
the assumption \eqref{ass:Loj-est} is satisfied for a sufficiently small $\eps=\eps(d)>0$.

To this end we want to exploit the pointwise bound on the holomorphic derivative $u_z$ obtained in formula (12) of
the proof of Lemma 1 in \cite{Topping-rigidity} which tells us that
 for any $z\in \C$
\newcommand{\dvcw}{dv_\C(w)}
\beq
\label{est:uz-initial}
\abs{u_z(z)}\leq (2\pi)^{-1}\int_\C \frac{\abs{\tc(u)(w)}}{\abs{z-w}}\dvcw+(2\pi)^{-1}\int_\C\frac{\abs{u_z}^2}{\abs{z-w}} \dvcw.
\eeq
Here $\tc(u)$ denotes the tension of the map $u:\C\to S^2$ with respect to the standard metric $g_{\C}=dx_1^2+dx_2^2$ and we point out that our notation deviates slightly from the one used in \cite{Topping-rigidity} and that here $\abs{u_z}^2$ denotes the size of the holomorphic derivative of the map $u$ into $S^2$ (rather than of the map $\tilde u$ into $\C$ as in \cite{Topping-rigidity}). 
 
The main difficulty in deriving the claimed estimate \eqref{claim:Loj-hol-energy-1} from the above formula is that we want to establish bounds that involve the tension with respect to a metric $g$ whose conformal factor can be highly concentrated at multiple points. 

\begin{proof}[Proof of Proposition \ref{prop:1}]
Let $p_i\in S^2$, $r_i$ and $u:S^2\to S^2$ be as in the proposition, where we note that after a rotation of the domain we can assume without loss of generality that the points $p_i$ all  have distance at least $\tilde c=\tilde c(d)>0$ from the equator $S^1\times \{0\}\subset S^2$. 

By symmetry it then suffices to bound the $z$-energy of $u:S^2\to S^2$ on the upper hemisphere $S^+=S^2\cap\{y_3>0\}$ and we will do so by working in the usual sterographic coordinates on the domain which map $S^+$ to the unit disc $\DD_1$ around $z=0$.

We now let $c=c(d)\in (0,\half)$ be so that points in $\pi(\DD_{(1+2c)^{-1}}\setminus\DD_{1-2c})\subset S^2$, have distance no more than $\tilde c$ from the equator for the constant $\tilde c=\tilde c(d)>0$ obtained above. Hence the centres $p_i\in S^2$ of the weights $\rho_{p_i,r_i}^2$ that are used to define the domain metric all correspond to points $b_i:=\pi^{-1}(p_i)$ outside of this annulus and we reorder these points so that 
$\abs{b_i}\leq 1-2c$ for $i\leq m$ for some $m\in\{0,\dots,d\}$ while $\abs{b_i}\geq (1-2c)^{-1}\geq 1+2c$ otherwise. 

In stereographic coordinates our weighted metric is given by 
$$g=\rho^2g_{\C} =(1+\sum\rho_{p_i,r_i}^2\circ \pi) \pi^* g_{S^2}$$ 
and we note that
 on $\DD_{1+c}$ the weights that correspond to $i>m$ are bounded by a constant $C=C(d)$  as the distance of these $p_i$ to $\pi(\DD_{1+c})$ is bounded away from zero. 

On the other hand 
we can use that the conformal factors $\rho_{p_i,r_i}\circ \pi$ with $i\leq m$   are bounded by 
$\rho_{p_i,r_i}\circ \pi\leq C\frac{\mu_i}{1+\mu_i^2\abs{x-b_i}^2}$ on bounded subsets of $\C$
and that the corresponding dilation factors $\mu_i$ in 
$M_{p_i,r_i}(p)=\pi_{p_i}(\mu_i\pi_{p_i}^{-1}(p))$ 
are of order $\mu_i\sim r_i^{-1}$, compare  Lemma \ref{lemma:dilation}.
In the following we can thus use that the conformal factor $\rho$ of our weighted metric $g$ is bounded by 
\beq
\label{est:conf-factor-prop}
\rho\leq C +C \sum_{i=1}^m \rho_i \text{ for } \rho_i:= \frac{\mu_i}{1+\mu_i^2\abs{x-b_i}^2} \text{ on  }\DD_{1+c}.
\eeq
We will show that 
\beq
\label{claim:Loj-hol-energy}
 E_{\partial}(u, \DD_1)\leq C(1+\log\bar \mu) \norm{\tau_g(u)}_{L^2(S^2,g)}^2+\tfrac{1}{4} E_{\partial}(u,\C)
 \eeq
 for $\bar \mu=\max\mu_i$ and for maps $u$ for which
\beq
\label{ass:proof-prop}
\La_g E_\partial(u)\leq \eps \text{ for } \La_g:= \prod_i (1+\log\mu_i)
\text{ and a suitably chosen } \eps=\eps(d)>0.
\eeq
We note that this is sufficient to establish \eqref{claim:Loj-hol-energy-1} for maps satisfying \eqref{ass:Loj-est}, and hence to complete the proof of the proposition,  since $\mu_i\sim r_i^{-1}$ and since \eqref{claim:Loj-hol-energy} applied to $u(\frac 1z)$  yields the analogue bound also on $E_\partial(u,\C\setminus \DD_1)=E_\partial(u(\frac1{\cdot}),\DD_1)$.

To prove \eqref{claim:Loj-hol-energy} it is convenient 
to introduce the 
 short hand $\ftau:= \rho^{-1}\abs{\tc(u)}$ to write the norm of the weighted tension as
$\norm{\tg(u)}_{L^2(\Om,g)}=\norm{\rho^{-1}\tc(u)}_{L^2(\Om,g_\C)}=\norm{\ftau}_{L^2(\Om,g_\C)}$.

For $z\in \DD_1$ we then use \eqref{est:uz-initial} and $\abs{\tau_\C(u)}=\rho f_\tau\leq C +C\sum_{1\leq i\leq m} \rho_i f_\tau$ to estimate 
\beqa\label{est:uz-basic}
\abs{u_z(z)}&\leq \int_{\DD_{1+c}} \frac{\rho \ftau}{\abs{z-w}} \dvcw 
+ \int_{\DD_{1+c}} \frac{\abs{u_z}^2}{\abs{z-w}}\dvcw+C\norm{\tc(u)}_{L^1(\C,g_\C)}+C \Ede(u)\\
&\leq C \sum_{0\leq i\leq m} h_i(z)+Cs(z) +C\norm{\tg(u)}_{L^2(S^2,g)}+C \Ede(u)
\eeqa 
for 
\beq
\label{def:hi}
h_0(z):=  \int_{\DD_{1+c}} \frac{\ftau}{\abs{z-w}} \dvcw \text{ and } h_i(z):=  \int_{\DD_{1+c}}\rho_i \frac{\ftau}{\abs{z-w}} \dvcw,\quad i=1,\ldots, m
\eeq
as well as  
\newcommand{\jj}{\mathit{s}} 
\beq\label{def:e}
\jj(z):= \int_{\DD_{1+c}} \frac{\abs{u_z}^2}{\abs{z-w}}\dvcw.
\eeq
Here we have furthermore used that the $L^1$ norm of the tension is conformally invariant and 
hence that 
$\norm{\tc(u)}_{L^1(\C,g_\C)}=\norm{\tg(u)}_{L^1(S^2,g)}\leq C \norm{\tg(u)}_{L^2(S^2,g)}
$ for $C=C(d)$ since $\text{Area}(S^2,g)\leq 4\pi (d+1)$. 

To obtain the desired estimate 
\eqref{claim:Loj-hol-energy}
on $\norm{u_z}_{L^2(\DD_1)}^2$ 
it hence suffices to obtain suitable bounds on $\norm{h_i}_{L^2(\DD_1)}$ and $\norm{\jj}_{L^2(\DD_1)}$. 
As in \cite{Topping-deg-1} we can view these functions $h_i$ and $s$ as Riesz-potentials 
$$R(\tilde f)(z):=\int_\C \frac{\tilde f(w)}{\abs{z-w}} \dvcw$$
for suitable $\tilde f:\C\to \R$. We however note that while
\beq\label{est:Riesz}
\norm{R(\tilde f)}_{L^{\frac{2q}{2-q}}(\C)}\leq C_q \norm{\tilde f}_{L^q(\C)} \text{ for every }\tilde f\in L^q(\C),
\eeq
whenever $q\in (1,2)$, this estimate is not valid for the one exponent $q=1$ for which the $L^q$ norm of the tension is conformally invariant.  

If $\rho$ was uniformly bounded, as was the case in \cite{Topping-deg-1}, this would be no problem as we could simply apply the above estimate for some $q\in (1,2)$ and then use that $\norm{\rho f_\tau}_{L^q}\leq C\norm{\rho}_{L^\infty} \norm{f_\tau}_{L^q}\leq C\norm{\rho}_{L^\infty} \norm{\tau_g(u)}_{L^2}$. 

Here we however deal with a situation where $\rho$ can be highly concentrated and where not only $\norm{\rho}_{L^\infty}$, but also $\norm{\rho}_{L^r}$ for any $r>2$, scales like a power of $\bar \mu$. 
If we were to use \eqref{est:Riesz} directly for some $q>1$ to bound the norms of $h_i$ we would hence end up with a \Loj estimate that involves a factor which scales like a power of $\bar \mu$ rather than like $\log(\bar\mu)$. Such an estimate would not be optimal and, importantly, would not allow us to prove the quantitative estimates with sharp rates as stated in Theorems \ref{thm:3} and \ref{thm:gauge}. 

We thus need to proceed with more care and we will split 
$\DD_{1}$ into subsets $A_J$, $J\in \N^m$, on which the conformal factors $\rho_i$ are essentially constant and will use that these domains are small whenever $\rho_i$ is large. 
We will obtain most of these domains as intersections of diadic annuli  $\DD_r\setminus \DD_{r/2}(b_i)$ 
for radii that range from order $\mu_i^{-1}$ to order $1$. To be more precise, we let 
$j_0=j_0(d)$ be so that $\frac{c}{4}\leq 2^{-j_0}\leq \frac{c}{2}$ for the constant $c=c(d)>0$ obtained above and for each 
$i=1,\ldots, m$ let $N_i\in \N$ be so that $2^{-N_i}\in [\mu_i^{-1},2\mu_i^{-1})$. For each $i$ we then partition $\DD_1$ into the sets $A_j^i$, $j=j_0,\ldots,N_i$, where 
\beqa \label{def:Ai}
 A_{j}^i:= \DD_{2^{-j}}(b_i)\setminus \DD_{2^{-j-1}}(b_i) \text{ for } j=j_0+1,\ldots ,N_i-1
\eeqa
while
\beq
\label{def:Ai-zero}
A_{N_i}^i:=\DD_{2^{-N_i}}(b_i) \text{ and }
A_{j_0}^i:=\DD_1(0)\setminus \DD_{2^{-(j_0+1)}(b_i)}.\eeq
We also  set  $A_j^i=\emptyset$ for all other $j$ and for $J=(j_1,\ldots, j_m)\in \N^m$ let $$A_J:=\bigcap A_{j_i}^i \text{ and } r_J:=\min(r_{j_i})=2^{-\max j_i} \text{ for } r_{j_i}:=2^{-j_i}. $$ 
In the following we use the convention that
we only ever sum over indices $J$ for which $A_J$ is non-empty and that any statements involving the sets $A_J$ or $A_i^j$ are to  be understood as statements about such non-empty sets. 

We note that these sets $A_J$ have the following basic properties: \\ The total number of sets $A_J$ is bounded by $\prod (N_i+1)\leq C\La_g$, $\La_g$ as in \eqref{ass:proof-prop}, and the diameter of each $A_J$  is bounded by $C r_J$ for a constant $C$ that could be chosen as $C=2$ for indices $J\neq (j_0,\ldots,j_0)$, but that we allow to depend on $d$ so that we can also include the case that $ J= (j_0,\ldots,j_0)$. 

We also note that since we work with diadic annuli we know that for any of the sets  $A_{j_0}^{i_0}$ and any $i\neq i_0$ there can be at most $3$ indices $j$ for which $A_i^j\cap A_{j_0}^{i_0}\neq \emptyset$ and for which $r_j\geq r_{j_0}$, i.e.~$j\leq j_0$. 
%
For any $\beta>0$ and any $i_0$ and $k_0$ we hence deduce that 
\beqa \label{est:sum-rJ}
 \sum_{J: j_{i_0}=k_0} r_J^\beta
&\leq \sum_{j\geq k_0} r_j^\beta\cdot \#\{J: j_i\leq j \text{ for } i\neq i_0 \text{ and } j_{i_0}=k_0 \} \\
&\leq r_{k_0}^\beta\sum_{j\geq k_0} 2^{-(j-k_0)\beta}(j-k_0+3)^{m-1} \leq C r_{k_0}^{\beta}
\eeqa
for some $C=C(d,\beta)$, and hence also 
\beq
\label{est:sum-rJ-total}
 \sum_{J} r_J^\beta\leq C=C(d,\beta).
 \eeq 
The main step in the proof of Proposition \ref{prop:1} is now to show
that 
\beq \label{def:Iq} 
I_q(u):= \sum_J r_J^{\gamma}\norm{u_z}_{L^{\frac{2q}{2-q}}(A_J)}^2 \text{ for } \gamma:=4-\tfrac4q \text{ and } q\in (1,2)
\eeq
is bounded by 
\beq
\label{claim:Iq}
I_q(u)\leq C \log\bar\mu \norm{\tg(u)}_{L^2(S^2,g)}^2+C\La_g\Ede(u)^2+C \Ede(u) I_{q}(u)
\eeq
for some $C=C(d,q)$.  We will prove this claim below and first explain how it implies the estimate \eqref{claim:Loj-hol-energy} that we need to establish to complete the proof of the proposition. 

To this end we can fix an arbitrary $q\in (1,2)$, say $q=\frac32$, and use that 
$\Ede(u)\leq \frac{\eps}{\La_g}\leq \eps$ is assumed to be small to obtain from \eqref{claim:Iq} that 
\beqs
I_q(u)\leq C \log\bar\mu \norm{\tg(u)}_{L^2(\C,g)}^2+C\La_g\Ede(u)^2 \text{ for some } C=C(d).
\eeqs
As  $\text{diam}(A_J)\leq C r_J$ and as $\gamma= 2(1-\frac{2-q}{q})$
we can thus bound
\beqas
\norm{u_z}_{L^2(\DD_1)}^2&=\sum\norm{u_z}_{L^2(A_J)}^2\leq \sum\norm{u_z}_{L^{\frac{2q}{2-q}}(A_J)}^2 \abs{A_J}^{1-\frac{2-q}{q}}\leq C I_q(u)\\
&\leq C \log\bar\mu \norm{\tg(u)}_{L^2(S^2,g)}^2+C\La_g\Ede(u)^2\leq C\log\bar\mu \norm{\tg(u)}_{L^2(S^2,g)}^2+C\eps\Ede(u)\\
&\leq C\log\bar\mu \norm{\tg(u)}_{L^2(S^2,g)}^2+\tfrac14\Ede(u)
\eeqas 
where the last estimate holds provided 
 $\eps=\eps(d)>0$ is chosen small enough. This gives \eqref{claim:Loj-hol-energy} and hence reduces the proof of the proposition to showing that $I_q(u)$ satisfies \eqref{claim:Iq}.

To prove \eqref{claim:Iq} we will analyse the contributions of each of the terms in \eqref{est:uz-basic} to $I_q$ separately. In these estimates we will use several times that the exponent $\gamma= 4-\frac4q$ is so that 
$2-\gamma=2\frac{2-q}{q}$ and hence so that 
for any  $\Om\subset \C$ and $f:\Om\to \R$
 \beq \label{est:f-sobolev-by-infty} 
 \norm{f}_{L^{\frac{2q}{2-q}}(\Om)}^2\leq C\text{diam}(\Om)^{2-\gamma}\norm{f}_{L^\infty(\Om)}^2
\text{ and }
 \norm{f}_{L^{q}(\Om)}^2\leq C \text{diam}(\Om)^{2-\gamma}\norm{f}_{L^2(\Om)}^2
 \eeq 
where here and in the following norms are computed with respect to the standard volume element of $\C$ unless indicated otherwise. 

Combining \eqref{est:sum-rJ-total} (applied for $\beta=2-\gamma>0$) with the first of these estimates ensures that 
the contribution of the last two terms in \eqref{est:uz-basic} to $I_q(u)$ is  bounded by
$ C\norm{\tg(u)}_{L^2(S^2,g)}^2+
C \Ede(u)^2$. 
It thus remains to control the contributions  to $I_q$ of the functions $h_i$, $i=0,\ldots, m$ and $\jj$ defined in \eqref{def:hi} and \eqref{def:e}.

For $h_0= R(\ftau \chi_{\DD_{1+c}})$ we can first apply 
 \eqref{est:sum-rJ-total} and then use 
\eqref{est:Riesz} to bound
\beqs
\sum_J r_J^\gamma\norm{h_0}_{L^{\frac{2q}{2-q}}(A_J)}^2\leq C \norm{h_0}_{L^{\frac{2q}{2-q}}(\DD_{1})}^2
\leq C\norm{\ftau}_{L^q(\DD_{1+c})}^2\leq C \norm{\ftau}_{L^2(\DD_{1+c})}^2\leq  C\norm{\tg(u)}_{L^2(S^2,g)}^2.
\eeqs 
For $i\geq 1$ we instead first use 
\eqref{est:sum-rJ} as well as that $A_J\subseteq A_{j_i}^i$ and $r_J\leq r_{j_i}$ to bound 
\beq
\label{est:hi-1}
\sum_J r_J^{\gamma}\norm{h_i}_{L^{\frac{2q}{2-q}}(A_J)}^2 \leq C \sum_{k=j_0}^{N_i} r_k^{\gamma} \norm{h_i}_{L^{\frac{2q}{2-q}}(A_k^i)}^2
\eeq
and will show that each of these terms is bounded by a fixed multiple of $\norm{\tau_g(u)}_{L^2(S^2,g)}^2$. 

To see this we consider the sets $B_k^i$ that are defined by
$B_k^i:= A_{k-1}^i\cup A_{k}^i\cup A_{k+1}^i$
for $k\in\{j_0+1,\ldots, N_i\}$ and   
$B_{j_0}^i:= (\DD_{1+c}\setminus \DD_1)
\cup A_{j_0}^i\cup A_{j_0+1}^i$ which are chosen so that
\beq 
\abs{z-w}^{-1}\leq C r_k^{-1} \text{ whenever }z\in A_k^i \text{ and } w\notin B_k^i.
\eeq
For $z\in A_k^i$
we can hence bound 
\beqa
\label{est:hi}
\abs{h_i(z)}&\leq  \norm{\rho_i}_{L^\infty(B_k^i)} R(\ftau \chi_{B_k^i}) +Cr_k^{-1}\sum_{j: \abs{j-k}\geq 2} \norm{\rho_i}_{L^\infty(A_j^i)}\norm{\ftau}_{L^1(A_j^i)} \\
&\leq \norm{\rho_i}_{L^\infty(B_k^i)} R(\ftau \chi_{B_k^i}) +C r_k^{-1} \sum_j 
 \norm{\rho_i}_{L^\infty(A_j^i)} r_j \norm{\ftau}_{L^2(A_j^i)}\\
&\leq C r_k^{-1} R(\ftau \chi_{B_k^i}) +C r_k^{-1} \sum_j 
2^{-(N_i-j)} \norm{\ftau}_{L^2(A_j^i)}
\eeqa 
where we recall that $A_{j}^i\neq \emptyset$ implies $j\leq N_i$ and  where
the last step follows since 
\beq
\label{est:rhoi-prop}
\rho_i(z)= \frac{\mu_i}{1+\mu_i^2\abs{z}^2}\leq C (\mu_ir_j)^{-1}r_j^{-1} \leq C 2^{-(N_i-j)} r_j^{-1} \text{ on } A_j^i.
\eeq 
Combining \eqref{est:hi}
 with \eqref{est:Riesz}, \eqref{est:f-sobolev-by-infty} and $\text{diam}(A_k^i)\leq \text{diam}(B_k^i)\leq C r_k$  yields 
\beqas
r_k^\gamma \norm{h_i}_{L^{\frac{2q}{2-q}}(A_k^i)}^2
&\leq C r_k^{\gamma-2} \norm{\ftau}_{L^q(B_k^i)}^2+C r_k^{\gamma-2} \diam(A_k^i)^{(2-\gamma)} \big(\sum_j 
2^{-(N_i-j)}\norm{\ftau}_{L^2(A_j^i)}\big)^2\\
&\leq C\norm{\ftau}_{L^2(B_k^i)}^2+C
\norm{\ftau}_{L^2(\DD_{1+c})}^2
\leq C\norm{\ftau}_{L^2(\DD_{1+c})}^2\leq\norm{\tau_g(u)}_{L^2(S^2,g)}^2 .
\eeqas 
Inserting this into 
\eqref{est:hi-1} and using that $N_i\leq C(1+\log\mu_i)$ hence gives the desired bound on the contribution of $h_i$ to $I_q$ of 
\beq
\label{est:hi-int-final}
\sum r_J^{\gamma}\norm{h_i}_{L^{\frac{2q}{2-q}}(A_J)}^2\leq  C(1+\log\mu_i) \norm{\tau_g(u)}_{L^2(S^2,g)}^2.
\eeq
Similarly, to bound the contribution of $\jj$ to $I_q(u)$, we consider the subsets 
$B_J:= \bigcap B_{j_i}^i$,  for $B_j^i$ as above, and use that $\abs{z-w}^{-1}\leq Cr_J^{-1}$ whenever $z\in A_J$ and $w\notin B_J$. 
For any $z\in A_J$ we can thus bound 
$$
\jj(z)\leq R(\abs{u_z}^2 \chi_{B_J})+Cr_J^{-1} \Ede(u)$$
which results in an estimate of
\beqas
\sum_{J}r_J^\gamma \norm{s}_{L^{\frac{2q}{2-q}}(A_J)}^2
&\leq C \sum_{J}r_J^\gamma \norm{\abs{u_z}^2}_{L^{q}(B_J)}^2+\sum_J
r_J^{\gamma-2}\diam(A_J)^{2-\gamma}
\Ede(u)^2\\
&\leq C\sum_Jr_J^\gamma \norm{\abs{u_z}^2}_{L^{\frac{q}{2-q}}(B_J)}\norm{\abs{u_z}^2}_{L^1(B_J)}+C\#\{J: A_J\neq \emptyset\} \Ede(u)^2\\
&\leq C \sum_Jr_J^\gamma \norm{u_z}_{L^{\frac{2q}{2-q}}(B_J)}^2 E_\partial(u) +C\La_g\Ede(u)^2\\
&
\leq C\Ede(u) I_q(u)+C\La_g\Ede(u)^2.
\eeqas
Altogether this establishes \eqref{claim:Iq} which completes the proof of the proposition. 
\end{proof}

 \appendix
 \section{Appendix}
 \subsection{Basic properties of M\"obius transforms} $ $\\
 Throughout the paper we use the following basic properties of dilations
\begin{lemma}
\label{lemma:dilation}
Let $M:S^2\to S^2$ be the M\"obius transform that corresponds to a dilation around a point $\bar x$ which maps $B_{\bar r}(\bar x)$ to a hemisphere for some radius $\bar r\in (0,\frac\pi 2]$. 
Then the corresponding conformal factors $\rho_M=\frac{1}{\sqrt{2}} \abs{dM}$ and $\rho_{M^{-1}}=\frac{1}{\sqrt{2}} \abs{dM^{-1}}$ satisfy
\beq
\label{est:rho-approx}
\rho_M(x)\sim \frac{(\bar r)^{-1}}{1+((\bar r)^{-1}d_{S^2}(x,\bar x))^2} \text{ and } \rho_{M^{-1}}(y)\sim \frac{(\bar r)^{-1}}{1+((\bar r)^{-1}d_{S^2}(y,-\bar x))^2}
\eeq 
In particular, there exists a universal constant $K_0$ so that
\beq
\label{est:rho-uniform}
\sup_{x_{1,2}\in B_{2\bar r}(x)} \frac{\rho_{M}(x_1)}{\rho_M(x_2)}\leq K_0 \text{ for all } x\in S^2.
\eeq 
Furthermore, for any $x\in S^2$ with $\bar r\leq \dist(x,\bar x)\leq \pi/2$ we have 
\beq
\label{rel:dist-Mob} 
\dist(M(x),-\bar x)\sim \frac{\bar r}{\dist(x,\bar x)}.
\eeq
\end{lemma}
Here we use the short hand $T_1\sim T_2$ for terms that are related by   $C^{-1}T_1\leq T_2\leq CT_2$ for a universal constant $C$.

\gcmt{probably don't even need a proof of this... }
\begin{proof}
We can assume without loss of generality that $\bar x=(0,0,1)$ and work in standard stereographic coordinates on both the domain and the target in which $M=\pi\circ \tilde M\circ \pi^{-1}$ is given by the dilation $\tilde M(x)=\mu x$ for some $\mu\sim (\bar r)^{-1}$. As $\abs{d\pi(x)}=c\frac{1}{1+\abs{x}^2}$ and as $\pi$ is conformal we hence obtain 
that
$$\rho_M(x)\sim 
 \abs{d\pi(\mu \pi^{-1}(y))}
\mu  \abs{d\pi^{-1}(y)} \sim \tfrac{\mu}{1+\mu^2 \abs{\pi^{-1}(y))}^2} 
\abs{d\pi(\pi^{-1}(y))}^{-1}
= \mu\tfrac{1+\abs{\pi^{-1}(y))}^2}{1+\mu^2 \abs{\pi^{-1}(y))}^2}. 
$$
The first claim immediately follows as  $\abs{\pi^{-1}(y)}\sim d(y,\bar x)\leqs 1$ for points $y$ in the upper hemisphere while $\abs{\pi^{-1}(y)}\geq 1$ and $ d(y,\bar x)\sim 1$ for points in the lower hemisphere. 
The second claim is then a simple consequence of the first as $M^{-1}$ corresponds to a dilation around the antipodal point with the same factor $\mu$. The third claim is also an immediate consequence of the first and the final claim can then again just be checked in stereographic coordinates. \end{proof}

\subsection{Proof of Lemma \ref{lemma:rational}}
\label{appendix:rational}
$ $\\
Let $h$ be a rational map and let 
$B_{R_i}(z_i)$, $i\in J_1$ respectively $i\in J_2$ be collections of balls for which the assumptions of Lemma \ref{lemma:rational} hold, where we can of course assume that $R_i\leq \pi$ and that  $d\in (0, \half]$, compare Remark \ref{rmk:rho-convention}. We let $\Om:= S^2\setminus \bigcup_{J_2} B_{4R_i}(z_i)$ be the set on which we claim to have the pointwise bound \eqref{claim:lemma-rat} on the energy density and as above use the short hand 
 $T_1\leqs T_2$ if  $T_1\leq CT_2$ for a universal constant and the analogue convention for $\sim$ and for $O(\cdot)$ notation. 

The smallness of the energy on the annuli $B_{2R_i}\setminus B_{R_i}(z_i)$, $i\in J_2$, allows us to choose  radii $\hat R_i\in [R_i,2R_i]$ so that the angular energy on the corresponding circles is small, namely so that 
$\int_{S^1}\abs{h_\theta(\hat R_i e^{i\theta})}^2 d\theta\leqs \eps_1$.
The oscillation of $h$ over these circles is hence of order $O(\sqrt{\eps_1})$. As $\hat R_i\leq 2R_i$ ensures that $\dist(\Om ,\partial B_{\hat R_i}(z_i))\geq 2R_i$ for all $i\in J_2$ we can also use that 
\beq 
\label{est:hallo-0}
 \abs{y-w} \sim \abs{y-z_i} \text{ for all } w\in \partial B_{\hat R_i}(z_i), y\in\Om \text{ and } i\in J_2.\eeq
We set $F:=\bigcup_{J_2} B_{\hat R_i}(z_i)$ and note that since $F$ always contains the second collection of balls  we still get the analogues of \eqref{ass:lemma-rat-1} and \eqref{ass:lemma-rat-3} if we cut out this set instead of the balls $B_{R_i}(z_i)$, $i\in J_2$. 

We first prove the claim \eqref{claim:lemma-rat} for points 
$y\in \Om$ which are contained in  $\bigcup_{i\in J_1} B_{4R_i}(z_i)$.
For each such point there is a 
 $j\in J_1$ so that \eqref{ass:lemma-rat-2} holds
 and hence so that 
$$ E(h, B_{dR_j}(y)\setminus F)\leq 2\eps_1 \text{ and  } \abs{y-z_j}\leq d^{-1} R_j.$$
We can then choose a radius  $\hat r\in [\half dR_j,dR_j]$ so that 
the angular energy of $h$ on $(\partial B_{\hat r}(y))\setminus F$ is of order  $O(\eps_1)$,  and hence so that  
 the oscillation of $h$ over any connected component of 
this set $(\partial B_{\hat r}(y))\setminus F$ is of order  $O(\sqrt{\eps_1})$.

We now consider the connected component $U_y$ of the set $ B_{\hat r}(y)\setminus F$ which contains $y$ and note that  $\partial U_y$ is made up of arcs of the circles $\partial B_{\hat R_i}(z_i)$, $i\in J_2$, and of connected components of $(\partial B_{\hat r}(y))\setminus F$. As the total number of balls in our collections is no more than $k$ we hence
know that the oscillation over every connected component of $\partial U_y$ is bounded by 
$C\sqrt{\eps_1}$ for a number $C=C(k)$ that only depends on $k$. 

This ensures that if we fix $c_0=c_0(k)>0$ sufficiently small and if necessary reduce $\eps_1$ then there must be at least some ball $B_{2c_0}(p_0)$ in $S^2$ which is disjoint from the image $h(\partial U_y)$ of all these curves. 
At the same time \eqref{ass:lemma-rat-2} ensures that $\Area(h(U_y))\leq E(h,U_y)\leq 2\eps_1$ will be strictly smaller than  $\Area(B_{c_0}(p_0))$
if $\eps_1=\eps_1(k)>0$ is small enough. 

We can hence pick 
 some $p_1\in B_{c_0}(p_0)\setminus h(U_y)$ so that  
\beq\label{est:hallo}
h(\partial U_y)\subset S^2\setminus B_{c_0}(p_1) \text{ and } p_1\notin h(U_y).\eeq
To prove the desired bound on $\abs{\na h(y)}$ we now introduce 
stereographic coordinates on the domain which are centred at $y$ and stereographic coordinates on the target which are centred at $-p_1$, and hence map the point $p_1$ which is disjoint from $h(U_y)$ to $\infty$.

In these coordinates $h$ is given by $f(z)=\pi_{-p_1}^{-1}\circ h\circ \pi_y(z)$, $\pi_p^{-1}$ the stereographic projection which maps $p$ to $0$ and $-p$ to $\infty$, and \eqref{est:hallo} ensures that this function is not just meromorphic, but holomorphic in $\tilde U_y=\pi_y^{-1}(U_y)$ and so that $\abs{f(z)}\leq C=C(k)$ on $\partial \tilde U_y=\pi_y^{-1}(\partial U_y)$.

We can hence apply the
Cauchy formula for derivatives to bound the derivative $\partial_z f(z)$ of the holomorphic function $f$ at the point $z=0$ which represents $y$ in these coordinates by 
\beq
\label{eq:Cauchy}
\abs{\partial_z f(0)}=\babs{\frac1{2\pi i}\int_{\partial \tilde U_y}  \frac{f(w)}{w^2}dw}\leq C \int_{\partial \tilde U_y} \abs{w}^{-2} dS
\eeq
where the first integral is to be understood as a contour integral over the connected components of $\partial U_y$ with the appropriately chosen orientations, while the second integral is a just a standard integral over a curve in the plane $ \C \cong \R^2$.  

Since $\hat r \leq dR_J \leq \frac\pi2$ we know that  $U_y\subset B_{\pi/2}(y)$ and $\tilde U_y\subset \DD_1$ are contained in sets where $d\pi_y^{-1}$ respectively $d\pi_y$ are uniformly bounded so can deduce from \eqref{eq:Cauchy} that 
\beqa
\abs{\na h(y)} &\leq C \abs{\partial_z f(0)} \leq  C \int_{\partial U_y} \dist_{S^2}(x,y)^{-2} dS_{g_{S^2}}(x)\\
& \leq \sum_{J_2} \int_{\partial B_{\hat R_i}(z_i)} \dist_{S^2}(x,y)^{-2} dS_{g_{S^2}}(x)+\int_{\partial B_{\hat r}(y)}\dist_{S^2}(x,y)^{-2} dS_{g_{S^2}}(x)\\
&\leq C \sum_{J_2} \hat R_i \dist_{S^2}(z_i,y)^{-2}+ C R_j^{-1} \leq C \sum_{J_2} \rho_i+C\rho_j
\eeqa
where the penultimate step follows from \eqref{est:hallo-0} while the last step uses \eqref{est:rho-approx} as well as that $\dist(z_i,y)\geq 2 R_i$ for $i\in J_2$ and $y\in \Om$. 

Having hence established the claimed bound \eqref{claim:lemma-rat} for all $y\in \Om$ which are contained in $\bigcup_{I_1} B_{4R_i}(z_i)$ it remains to consider points $y$ so that $\dist(y,z_i)\geq 4R_i$ for all $i\in J_1\cup J_2$ and for such points we can modifying the above argument as follows: 
\\
Given any $i\in J_1$ we recall that the set of all indices $j$  for which 
\eqref{ass:lemma-rat-2} holds true for 
$x=z_i$ is not empty and note that it will always contain some index $j_i$ for which 
$R_{j_i}\leq R_i$. Indeed, if there are any $j$ with $R_j>R_i$ for which \eqref{ass:lemma-rat-2} holds then we can simply choose $j_i=i$.
As above we then use \eqref{ass:lemma-rat-2} to choose  
 a radius $\hat r_i\in [\half d R_{j_i}, dR_{j_i}]$ so that 
 the angular energy on $\partial B_{\hat r_{i}}(y_i)\setminus F$ is small.
We can then argue exactly as above except that we now work on the connected component $U_y$ of $(F\bigcup_{J_1}  B_{\hat r_{i}}(y_i))^c$ which contains $y$. In this case \eqref{est:hallo-0} 
is also true for indices of $J_1$ so we 
again obtain the desired bound of 
$$\abs{\na h( y)}\leqs \sum_{J_1\cup J_2} \hat r_i  \dist(y, z_i)^{-2}\leqs  \sum_{J_1\cup J_2} R_i  \dist( y, z_i)^{-2} \leqs \sum_{J_1\cup J_2} \rho_i
$$
where the last step follows from \eqref{est:rho-approx} and as we now consider points with $ \dist( y, z_i)\geq 4R_i$ for all $i\in J_1\cup J_2$.

M. Rupflin: Mathematical Institute, University of Oxford, Oxford OX2 6GG, UK\\
\textit{melanie.rupflin@maths.ox.ac.uk}

\end{document}